\documentclass[11pt,reqno,centertags,noamsfonts]{amsart}
\usepackage[margin=34mm]{geometry}

\usepackage[latin2]{inputenc}
\usepackage{amsmath}
\usepackage{amssymb}
\usepackage{amsthm}
\usepackage{xcolor}
\usepackage{color}
\usepackage[UKenglish]{babel}
\usepackage[colorlinks=true,linkcolor=blue,
anchorcolor=black,citecolor=cyan,filecolor=black,menucolor=black,runcolor=black,
urlcolor=blue]{hyperref}

\usepackage{mathtools}
\mathtoolsset{showonlyrefs}
\usepackage{comment}
\usepackage{enumitem}
\usepackage{csquotes}
\usepackage{mathrsfs}
\newcommand{\scrp}{\mathscr{P}}
 
\newcommand{\weakstar}{\overset{*}{\rightharpoonup}}
\newcommand{\wk}{\rightharpoonup}

\usepackage{physics} 

\newtheorem{theorem}{Theorem}[section]

\newtheorem{proposition}[theorem]{Proposition}
\newtheorem{lemma}[theorem]{Lemma}
\newtheorem{corollary}[theorem]{Corollary}
\newtheorem{definition}[theorem]{Definition}
\newtheorem{remark}[theorem]{Remark}

\newtheorem*{theorem*}{Theorem}
\newtheorem*{rem}{\normalfont\em Remark}

\allowdisplaybreaks[2]

\newcommand{\supp}{\operatorname{supp}} 
\newcommand{\divv}{\operatorname{div}}

\newcommand{\diag}{\operatorname{diag}}

\newcommand{\dist}{\operatorname{dist}}

\newcommand{\rel}{\mathrm{rel}}
\newcommand{\loc}{\mathrm{loc}}
\newcommand{\Om}{\Omega}
\newcommand{\ve}{\varepsilon}

\newcommand{\vp}{\varphi}
\newcommand{\ol}{\overline}

\newcommand{\uin}{u^\mathrm{in}}
\newcommand{\Uin}{U^\mathrm{in}}
\newcommand{\uvec}{U}
\newcommand{\uve}{u^{\ve}}
\newcommand{\surfnormal}{\textbf{n}_{\hypersurf}}
\newcommand{\hypersurf}{\mathcal{S}}
\newcommand{\hsurfCl}{\mathcal{\ol S}}

 \newcommand{\atype}{$\mathfrak{A}$}
  \newcommand{\btype}{$\mathfrak{B}$}
   \newcommand{\ctype}{$\mathfrak{C}$}
    \newcommand{\Atype}{\mathfrak{A}}
   \newcommand{\Btype}{\mathfrak{B}}
   \newcommand{\Ctype}{\mathfrak{C}}
   \newcommand{\FF}{\mathbb{Y}}
\newcommand{\WW}{\mathbb{W}}
 
 \newcommand{\sn}{{1+n}}
 \newcommand{\vd}{\alpha}
 \newcommand{\hh}{\hat h}
 \newcommand{\PP}{\hat\scrp}
 \newcommand{\hA}{\hat A}
 \newcommand{\MM}{\mathbb{\hat M}}
 
 \newcommand{\sumA}{\sum_{\alpha\in\mathfrak{A}}u_\alpha}

\numberwithin{equation}{section}

\makeatletter
\@namedef{subjclassname@2020}{\textup{2020} Mathematics Subject Classification}
\makeatother

\author[K. Hopf]{Katharina Hopf}
\address{Katharina Hopf, Weierstrass Institute for Applied Analysis and Stochastics (WIAS),
	Mohrenstrasse 39, 10117 Berlin, Germany, 
	E-Mail: hopf@wias-berlin.de}

\author[M. Burger]{Martin Burger}
\address{Martin Burger, Department of Mathematics, Friedrich-Alexander Universit\"at Erlangen-N\"urnberg, Cauerstrasse 11, 91058 Erlangen, Germany,
	E-mail: martin.burger@fau.de.}
	


\title[Multi-species diffusion with size exclusion
]{On multi-species diffusion with size exclusion
}

\date{}

\begin{document}
	
		\begin{abstract}	
			We revisit a classical continuum model for the diffusion of multiple species with size-exclusion constraint, which leads to a degenerate nonlinear cross-diffusion system. 
			The purpose of this article is twofold: first, it aims at a systematic study of the question of	 existence of weak solutions and their long-time asymptotic behaviour. Second, it provides a weak-strong stability estimate for a wide range of coefficients, which had been missing so far. 
			
			In order to achieve the results mentioned above, we exploit the formal gradient-flow structure of the model with respect to a logarithmic entropy, which leads to best estimates in the full-interaction case, where all cross-diffusion coefficients are non-zero. Those are crucial to obtain the minimal Sobolev regularity needed for a weak-strong stability result. For meaningful cases when some of the coefficients vanish, we provide a novel existence result based on approximation by the full-interaction case.
	\end{abstract}
	
	\maketitle
	
	\section{Introduction}	
	
	Nonlinear cross-diffusion models and their mathematical analysis have received  
	considerable
	attention in the last years, including various 
	degenerate cases as arising in models with size exclusion or local repulsion (cf.\ e.g.~\cite{berendsen2017cross,berendsen2020,BDiPS_2010,chen2021cross,chen2021rigorous,JS_2013,liero2013gradient,Juengel_2015,ZJ_2017}).
	While the local-in-time existence and uniqueness of classical solutions has been established a while ago in a more general setting (cf.~\cite{Amann_1990_II}), the questions of global-in-time regularity and qualitative properties of weak solutions 
	continue to pose a significant challenge due to non-uniform parabolicity and missing maximum principles for systems.
	From a modelling point of view such systems can naturally be derived from lattice systems for multiple species with size exclusion, which found applications in particular in cell biology (cf. \cite{burger2012nonlinear,noureen2021swapping,painter2009continuous,	simpson2009multi}). 
	
	In this paper we will revisit the canonical macroscopic cross-diffusion system arising from these models and derive new results both on the global existence of weak solutions as well as their stability and uniqueness (in the sense of a weak-strong stability result).
	For $T^*\in(0,\infty]$, $n\in\mathbb{N}$, a smooth bounded domain $\Om\subset\mathbb{R}^d$, $d\in\mathbb{N}$,
	and constant coefficients $K_{ij}$,  $i,j\in\{0,\dots,n\}$, 
	we consider the cross-diffusion system
	\mathtoolsset{showonlyrefs=false}
	\begin{subequations}\label{eq:sys.all}
	\begin{align}\label{eq:sys}  
			\partial_t u_i - \nabla \cdot \left[\sum_{j=0}^n K_{ij} (u_j\nabla u_i-u_i\nabla u_j)\right] = 0  \quad &\text{ in }(0,T^*) \times \Omega,\quad i\in\{0,\ldots, n\},\qquad
	\end{align}
	supplemented by no-flux boundary conditions
	\begin{align}\label{eq:sys.bc}
	\nu\cdot\sum_{j=0}^n K_{ij} (u_j\nabla u_i-u_i\nabla u_j) = 0 \quad &\text{ on }(0,T^*)\times \partial\Omega,\quad i\in\{0,\ldots, n\},\quad
	\end{align}
\end{subequations}\mathtoolsset{showonlyrefs}with $\nu$ denoting the outward unit normal to $\partial\Om$.
 Here, $u_i=u_i(t,x)\ge0$ denotes the density of the $i$-th species $X_i$.
Since effectively only off-diagonal entries of $(K_{ij})_{ij}$ enter into~\eqref{eq:sys},~\eqref{eq:sys.bc}, we may assume,  without loss of generality,  that $K_{ii}=0$ for all $i\in\{0,\dots,n\}$, a convention to be adopted throughout this paper.

Supposing the natural symmetry hypothesis $K_{ij}=K_{ji}$, the pointwise sum $\sum_{i=0}^nu_i$ of all densities is formally preserved under the evolution.
Assuming, for simplicity, that $\sum_{i=0}^n{u_i}_{|t=0}\equiv 1$, 
  we are thus interested in solutions $u=(u_0,\dots,u_n)$ of the above PDE system taking values in
	the planar hypersurface
	\begin{align}
		\hypersurf:= \left\{ u':=(u_0',\ldots,u_n') \in (0,\infty)^{n+1}:\; \sum_{i=0}^n u_i' = 1\right\}.
	\end{align}

Our study is based on the well-known observation (cf. \cite{BDiPS_2010,liero2013gradient,ZJ_2017}) that system~\eqref{eq:sys.all} can be written as a gradient flow of the convex integral functional
\begin{align}\label{eq:def.H}
	\mathcal{H}(u)=\int_\Om h(u)\,\dd x, \qquad h(u)=\sum_{i=0}^n \lambda(u_i)
\end{align}
with $\lambda(s):=s\log s-s+1$, 
that takes the form
	\begin{align}\label{eq:sys.gf}
	\begin{aligned}
		\partial_t u_i - \nabla \cdot \left(\sum_{l=0}^n\mathbb{M}_{il}(u)\nabla D_lh(u)\right)= 0  \quad &\text{ in }(0,T^*) \times \Omega,\quad i\in\{0,\ldots, n\}
	\end{aligned}
\end{align}
for a suitable symmetric and positive semi-definite mobility 
$\mathbb{M}:\hsurfCl\to \mathbb{R}^{(n+1)\times(n+1)}$.

To see this, we first rewrite equation~\eqref{eq:sys} more concisely as 
$$\partial_tu_i-\nabla\cdot(\sum_jA_{ij}(u)\nabla u_j)=0$$ with the matrix $A=(A_{ij})_{i,j=0,\dots,n}$ given by 
 \begin{align}\label{eq:defA}
 	A_{ij}(u) = \delta_{ij}\sum_{k=0}^nK_{ik}u_k-K_{ij}u_i.
 \end{align}
Here, $\delta_{ij}$ denotes the Kronecker symbol. Then, defining
 the \textit{mobility matrix} $\mathbb{M}(u):=A(u)(D^2h(u))^{-1}$,
 system~\eqref{eq:sys} turns into~\eqref{eq:sys.gf}.
For later reference, we note that the components of $\mathbb{M}(u)$ take the form
\begin{align}\label{eq:mobility}
	\mathbb{M}_{il}(u)=u_i(\delta_{il}\sum_{k=0}^nK_{ik}u_k - K_{il}u_l),
\end{align}
from which we see in particular that $\mathbb{M}(u)$ is symmetric for all $u\in\hsurfCl$ if and only if $(K_{ij})$ is.
Furthermore, as we will see in Section~\ref{ssec:matrix},  for $(K_{ij})$ symmetric
the positive semi-definiteness of $\mathbb{M}(u)$ for all $u\in\hsurfCl$
is equivalent to the non-negativity of $K_{ij}$ for all $i\not=j$.
This motivates imposing the following mild hypotheses on the coefficients $K_{ij}$.

\smallskip

	\paragraph{\bf Hypotheses}
	\begin{enumerate}[label=\textup{(H\arabic*)}]
		\item\label{it:K.sp} Symmetry:  $K_{ij}=K_{ji}$ for all $i,j\in\{0,\dots, n\}$, $i\not=j$.
		\item\label{it:K.nn} Non-negativity:  $K_{ij}\ge0$ for all $i,j\in\{0,\dots,n\}$, $i\not=j$.
	\end{enumerate}
 Note that from a microscopic point of view the $K_{ij}$ are proportional to the rate at which particles of type $i$ and $j$ change sites and hence are non-negative and symmetric. 
 In particular, hypotheses~\ref{it:K.sp} \&~\ref{it:K.nn} are fulfilled in all relevant applications and will therefore be assumed throughout this manuscript without necessarily being referred to.
 The symmetry of the $K_{ij}$ implies that, upon adding up the equations~\eqref{eq:sys}, 
$$ \frac{\dd}{\dd t}\sum_{i=0}^nu_i =\sum_{i=0}^n \partial_t u_i = 0, $$
and the non-negativity of $K_{ij}$ can be used to conclude preservation of non-negativity of each $u_i$, hence they are directly related to obtaining a solution in $\hsurfCl$.
More specifically, under hypotheses~\ref{it:K.sp},~\ref{it:K.nn}  and an additional non-degeneracy condition on $\{K_{ij}\}$ (such as~\ref{it:K.pos} or~\ref{it:ok} below), a strong maximum-type principle (cf.\ e.g.~\cite[Chapter~III]{DanersKMedina_1992},~\cite{QS_2019_book}) ensures that smooth flows starting in the surface $\hypersurf$ will not leave it as long as they maintain their regularity. 

Our analysis further requires some non-degeneracy condition on the coefficients $K_{ij}$.
While existence and equilibration will be established under a rather mild condition (to be specified in Section~\ref{sec:ex}, see~\ref{it:ok}), our stability analysis in Section~\ref{sec:stab} requires in addition to~\ref{it:K.sp} and~\ref{it:K.nn} the \textit{full-interaction} hypothesis
\begin{enumerate}[label=\textup{(H2$^*$)}]
	\item\label{it:K.pos} Positivity:  $K_{ij}>0$ for all $i,j\in\{0,\dots,n\}$, $i\not=j$.
\end{enumerate}
Note that from a microscopic point of view, the strengthened condition \ref{it:K.pos} means that all combinations of different particles can change sites and hence there is no situation where particles can get stuck.
In this context, let us recall that the diagonal values $K_{ii}$ are irrelevant for the dynamics and have been set to zero to simplify notations.

\paragraph{\bf Approach}
Exploiting the formal gradient-flow structure~\eqref{eq:sys.gf} is a key feature of our analysis.
At a formal level, \eqref{eq:sys.gf} implies the following basic entropy dissipation law
\begin{align}\label{eq:apriori}
	\frac{\dd}{\dd t}\mathcal{H}(u) + \int_\Om \frac{1}{2}\sum_{i,j=0}^nK_{ij}u_iu_j\left|\frac{\nabla u_i}{u_i}-\frac{\nabla u_j}{u_j}\right|^2\,\dd x = 0
\end{align}
along solutions $u=u(t,x)\in \mathcal{\ol S}$,
which is the main building block in our global existence and long-time asymptotic analysis. The latter will be carried out under a hypothesis far weaker than~\ref{it:K.pos} (cf.\ Section~\ref{sec:ex},~\ref{it:ok}).
To explain the new difficulties arising when relaxing~\ref{it:K.pos}, let us note  that hypothesis~\ref{it:K.pos} ensures that the dissipated quantity $\int_\Om D$, where
$$ D:=\frac{1}{2}\sum_{i,j=0}^nK_{ij}u_iu_j\left|\frac{\nabla u_i}{u_i}-\frac{\nabla u_j}{u_j}\right|^2,$$ 
provides $L^2$ control of the Sobolev gradient of $\sqrt{u_i}$, since then\footnote{In fact, by inequality~\eqref{eq:ffGrad} in Lemma~\ref{l:M.coerc} one has $D\ge4\min_{\{\alpha,\ell:\alpha\not=\ell\}}K_{\alpha\ell}\cdot\sum_{i=0}^n|\nabla\sqrt{u_i}|^2$.} 
$D\gtrsim \sum_{i=0}^n|\nabla\sqrt{u_i}|^2$.
Without hypothesis~\ref{it:K.pos} this bound is no longer valid and
the dissipation law~\eqref{eq:apriori} provides only poor gradient control of solutions. Hence, we cannot rely on standard Sobolev space theory and compactness arguments to pass to the limit in the approximation scheme. 
Indeed, taking for instance $K_{i0}>0$ and $K_{ij}=0$ for all $i,j\in\{1,\dots,n\}$, one even has 
\begin{align}
	\sum_{i,j=0}^nK_{ij}u_iu_j\left|\frac{\nabla u_i}{u_i}-\frac{\nabla u_j}{u_j}\right|^2
\le K\left(\frac{1}{u_0}|\nabla u_0|^2+u_0\sum_{j=0}^n\frac{1}{u_j}|\nabla u_j|^2\right),
\end{align}
where $K:=\max_{1\le i\le n}K_{0i}$, so that any gradient control of $u_j$, $j\ge1$, obtained from~\eqref{eq:apriori}
 is lost when $u_0$ vanishes. (This bound follows from an estimate analogous to that in the proof of Lemma~\ref{l:M.coerc}.)

For proving our weak-strong uniqueness and stability result
(with the relative entropy as a generalised distance), the dissipation property~\eqref{eq:apriori} is not sufficient. Here, we follow~\cite{Hopf_prepr2021} and rely on the structure~\eqref{eq:sys.gf} for controlling the relative entropy of a weak solution with respect to a strong solution. Developing further the approach in~\cite{Hopf_prepr2021}, we will see that 
weak-strong stability is mostly a consequence of the general structural properties of~\eqref{eq:sys.gf}: apart from the smoothness properties of $h$ and $\mathbb{M}$, the strict convexity\footnote{In this article, \textit{strict convexity} of the function $h\in C^2(\hypersurf)$ is to be understood in the sense of the Hessian $D^2h(u)$ being positive definite at every point $u\in \hypersurf$.} of the entropy, and the parabolicity of the gradient-flow evolution system~\eqref{eq:sys.gf}, 
we only rely on the a priori estimate~\eqref{eq:apriori} providing sufficient gradient control.
Precisely because of this last prerequisite
our stability analysis is confined to models $\{K_{ij}\}$ obeying the full-interaction hypothesis~\ref{it:K.pos}. However, it significantly improves the perturbative results of \cite{berendsen2020}, which are based on the same assumption, but also need all values $\{K_{ij}\}_{i\not= j}$ to be close to some positive constant.

A primary purpose of the present manuscript is to expose key aspects of the entropy method for (degenerate) cross-diffusion systems such as~\eqref{eq:sys.all} rather than aim for 
optimal estimates, and we would like to mention that, in the case of full interactions,
a version of the weak-strong stability estimate~\eqref{eq:stab.f} may be obtained under relaxed hypotheses on the strong solution and the initial data. In particular, using a suitably adapted relative entropy as recently introduced in~\cite{BRZ_preprint}, it might be possible to admit general non-negative initial data that are not necessarily strictly positive. 
Let us further note that our existence theory and
the weak-strong stability estimate are
 equally valid in the presence of external potentials, that is, for models~\eqref{eq:sys.gf} where the Boltzmann-type entropy $\sum_i\lambda(u_i)$ is enhanced by 
a linear potential part $\sum_{i}u_i\Upsilon_i$ 
 for given smooth potentials $\Upsilon_i:\ol\Om\to\mathbb{R}$.
This leads to evolution systems~\eqref{eq:sys} of the form 
\begin{align}
		\partial_t u_i - \nabla \cdot 
		\big[\sum_{j=0}^n K_{ij} (u_j\nabla u_i-u_i\nabla u_j + u_iu_j\nabla(\Upsilon_i-\Upsilon_j))\big] = 0.
\end{align}
An extension to mean-field interaction (cf.\ e.g.~\cite{berendsen2017cross}) also seems possible in certain regimes.
Let us finally mention a parallel development~\cite{HJT_preprint} covering Maxwell--Stefan diffusions in a non-degenerate regime (corresponding to the case of full interactions).

\subsection{Main results}
Our main results can be divided into two parts, contained in Sections~\ref{sec:stab} and~\ref{sec:ex}, respectively.
\begin{enumerate}[label=(\roman*)]
	\item\label{it:part.stab} The first part concerns fully interacting systems~\ref{it:K.sp}, \ref{it:K.pos}:
 Here, gradient bounds are available globally and existence of weak solutions is well established~\cite{Juengel_2015}. In this setting, we are able to perform a fairly general weak-strong stability and asymptotic analysis.
	\item\label{it:part.pint} The second part considers more general partially interacting systems.  These are quite relevant for applications, but due to the degeneracy of the parabolic structure even the issue of existence in a weak sense has not been fully understood so far. 
	We establish the existence of weak solutions (in a slightly generalised sense) as well as their convergence to the steady state in the long-time limit. 
\end{enumerate}

\subsubsection{Results  on full interactions in part~\ref{it:part.stab}}

Let $T^*\in(0,\infty]$.
\begin{definition}[Weak solution]\label{def:weak.sol}
	We call $u=(u_0,\dots,u_n)$ a weak solution of~\eqref{eq:sys.all} with  if $u_i\in L^2(0,T;H^1(\Om))$, $\partial_tu_i\in L^2(0,T;(H^1(\Om))^*)$ for all $T<T^*$ and $i\in\{0,\dots,n\}$, if  $u(t,x)\in \hsurfCl$ for a.e.\ $(t,x)\in(0,T^*)\times\Om$, and if 
	$u$ satisfies~\eqref{eq:sys.all} in the weak sense, i.e.\ if for all $T<T^*$, all 
	$\psi\in L^2(0,T;H^1(\Om))$ and all $i\in\{0,\dots, n\}$
	\begin{align}\label{eq:weak}
		\int_0^T\langle\partial_tu_i,\psi\rangle\,\dd t 
		+	\int_0^T\!\int_\Om \sum_{j=0}^nA_{ij}(u)\nabla u_j\cdot\nabla\psi\,\dd x\dd t =0.
	\end{align}
\end{definition}
Observe that under the hypotheses of Definition~\ref{def:weak.sol} one has $0\le u_i\le 1$ for all $i$.

For systems~\eqref{eq:sys.all} obeying hypothesis~\ref{it:K.pos}, we perform a weak-strong stability analysis choosing the relative entropy 
	\begin{align}\label{eq:def.H.rel}
	\mathcal{H}_\rel(u,\tilde u)&:=\int_\Om h_\rel(u,\tilde u)\,\dd x := \int_\Om \big[h(u)-\sum_{i=0}^nD_i h(\tilde u)(u_i-\tilde u_i)-h(\tilde u)\big]\,\dd x\quad
\end{align} 
as a generalised distance  between a weak solution $u$ and some strictly positive reference solution $\tilde u$ that enjoys certain extra regularity properties. 
Our most interesting result under the full-interaction hypothesis consists in a weak-strong stability estimate, where `strong solutions' are defined in the following way.

\begin{definition}[Strong solution]\label{def:strong.sol}
	We call $\tilde u=(\tilde u_0,\dots,\tilde u_n)$ a strong solution of~\eqref{eq:sys.all} if 
	$\tilde u\in C^{0,1}([0,T^*)\times\ol\Om)^{1+n}$,
	if $\tilde u(t,x)\in \hypersurf$ for all $(t,x)\in[0,T]\times\ol\Om$ and all $T<T^*$,
	and if $\tilde u$ satisfies \eqref{eq:sys.all} in the weak sense (as specified in Definition \ref{def:weak.sol}).
\end{definition}
Observe that our notion of a strong solution not only requires Lipschitz regularity, but also strict positivity of all components $\tilde u_i$ on $[0,T^*)\times\ol\Om$
and as a consequence $\min_{[0,T]\times\ol\Om}\tilde u_i>0$ for all $T<T^*$ and all $i\in\{0,\dots,n\}$.
Let us further remark that under hypotheses~\ref{it:K.sp},~\ref{it:K.pos} (and more generally under~\ref{it:K.sp},~\ref{it:K.nn} and the condition~\ref{it:ok} to be introduced in Section \ref{sec:ex})
 for given Lipschitz continuous initial data $\uin$ with $\uin(\ol\Om)\subset\hypersurf$, local existence of a unique strong solution follows from Amann~\cite{Amann_1990_II} and the parabolic structure of the system (see Section~\ref{sec:entropy} for details). 

Our weak-strong uniqueness and stability result states as follows.
\begin{theorem}[Weak-strong stability]\label{thm:uniq}
	Assume hypotheses~\ref{it:K.sp},~\ref{it:K.pos} and 
	suppose that $\tilde u\in C^{0,1}([0,T^*)\times\ol\Om)^{1+n}$  is a strong solution in the sense of Definition~\ref{def:strong.sol}.
	For all $T\in(0,T^*)$ there exists a constant $C_{T,\tilde u}<\infty$ only depending on
	$\|\tilde u\|_{C^{0,1}([0,T]\times\ol\Om)}$ and $(\inf_{[0,T]\times\ol\Om} \tilde u_i)_{i=0}^n$ such that the following holds true:
	
	any weak solution $u$ in the sense of Definition \ref{def:weak.sol} that has the additional regularity $\sqrt{u_i}\in L^2_\loc([0,T^*);H^1(\Om))$ for all $i \in \{0,\ldots,n\}$ obeys the stability estimate
	\begin{align}\label{eq:stab.f}
		\mathcal{H}_\rel(u(t),\tilde u(t))\le \mathcal{H}_\rel(u(0),\tilde u(0)) \exp(C_{T,\tilde u}t)\quad\text{ for al }t\in[0,T],
	\end{align}
	where $\mathcal{H}_\rel$ denotes the relative entropy functional~\eqref{eq:def.H.rel}.
\end{theorem}
If $u_{|t=0}=\tilde u_{|t=0}$,  the above theorem implies that $u=\tilde u$ a.e.\ in~$(0,T^*)\times\Om$.

See Section~\ref{ssec:stab.est} for the proof of Theorem~\ref{thm:uniq}. 
The regularity hypothesis $\sqrt{u_i}\in L^2_\loc([0,T^*);H^1(\Om))$ in Theorem~\ref{thm:uniq} is natural under hypothesis~\ref{it:K.pos}, 
and the existence of weak solutions obeying this hypothesis is well established, see Theorem~\ref{thm:ex.full} below.

For completeness, we further note the following asymptotic stability result, which  follows from classical arguments in the full-interaction regime.
\begin{theorem}[Exponential stability of steady states]\label{thm:stab.ss}
	Assume hypotheses~\ref{it:K.sp},~\ref{it:K.pos} and let $\ol u\in\hsurfCl$.
	There exist finite positive constants $\epsilon=\epsilon(\Om)$ and $C=C(\Om)$
	such that any weak solution $u=(u_0,\dots,u_n)$ of~\eqref{eq:sys.all} in the sense of Definition~\ref{def:weak.sol} 
	with the extra regularity $\sqrt{u_i}\in L^2_\loc([0,T^*);H^1(\Om))$ 
	and average values $\frac{1}{|\Om|}\int_\Om u_i(0)\,\dd x=\ol u_i$
	for all $i\in\{0,\dots,n\}$ satisfies the decay estimate 
	\begin{align}\label{eq:stab.ss}
		\|u(t)-\ol u\|_{L^1(\Om)}^2\le C\mathcal{H}_\rel(u(0),\ol u)\mathrm{e}^{-\epsilon \kappa t}\qquad\text{ for all }t>0,
	\end{align}
 where $\kappa:=\min_{i\not=j}K_{ij}>0$.
\end{theorem}
The proof of Theorem~\ref{thm:stab.ss} is given in Section~\ref{ssec:stab.ss}.

\subsubsection{Results in part~\ref{it:part.pint} on partial interactions}

Since the precise statement of our results concerning partial interactions requires several additional concepts and notations that will be introduced in detail in Section~\ref{sec:ex}, we here prefer to only summarise the main results on part~\ref{it:part.pint}:
\begin{itemize}
	\item Large-data global existence of weak solutions (see Theorem~\ref{thm:pex}),
	\item Convergence to equilibrium of the constructed solutions (see Theorem~\ref{thm:conv.c}).
\end{itemize}
While for specific subclasses of models existence has previously been obtained in the literature (see~\cite{BDiPS_2010,Juengel_2015,ZJ_2017}), our result concerning the long-time asymptotics appears to be the first rigorous one for partially interacting systems.

We should mention that for certain degenerate classes of partially interacting models
the regularity obtained from entropy estimates does not suffice to uniquely determine all
flux terms $K_{ij}(u_j\nabla u_i-u_i\nabla u_j)$
appearing in~\eqref{eq:sys}. In fact, in contrast to the models covered by existing literature~\cite{BDiPS_2010,Juengel_2015,ZJ_2017}, the present existence analysis faces situations where, due to poor gradient control, expressions of the form $K_{ij}u_j\nabla u_i$
may not even be defined in the distributional sense (despite the fact that $0\le u_l\le 1$ for all $l$). 
Hence, our existence result, Theorem~\ref{thm:pex}, is formulated for a somewhat generalised notion of weak solutions that, however, reduces to the standard concept of  weak or distributional solutions when restricting to the models considered in~\cite{BDiPS_2010,Juengel_2015,ZJ_2017}.
Besides, the generalised notion of weak solutions in Theorem~\ref{thm:pex} respects mass conservation and allows to uniquely determine the long-time asymptotic behaviour of solutions (cf.~Theorem~\ref{thm:conv.c}).

\subsection{Previous literature}\label{ssec:lit}
Below, we provide an overview of the analytical results available for system~\eqref{eq:sys.all}. 
We should point out that~\cite{BDiPS_2010} was the first work to analyse such a system (for $n=2$). 
\\[1mm]\noindent 
Literature on global existence:
\begin{itemize}
	\item Full interactions: \begin{itemize}
		\item the case $K_{ij}>0$ for all $i\not=j$ has been fully covered by \cite[Theorem~2]{Juengel_2015}, see~\cite{BE_2018} for details.
		\item \cite{BE_2018}: global existence of weak solutions for non-zero flux boundary conditions and time-dependent domains in one space dimension $d=1$. Equilibration for suitable data.
	\end{itemize}
	\item Partial interactions:
	\begin{itemize}
		\item \cite{BDiPS_2010}: global existence of weak solutions for $n=2$ and $K_{i0}>0,K_{ij}=0$ for $j,i\ge1$. 	(Only the case $d\in\{1,2,3\}$ was considered, but the proof does not rely on this restriction.)
	\item \cite[Theorem~3]{Juengel_2015}: global existence of weak solutions again for $n=2$ and $K_{i0}>0,K_{ij}=0$ for $j,i\ge1$, but more general transition rates (not considered in the present manuscript).
	\item \cite{ZJ_2017}: global existence of weak solutions for general $n\ge2$ and $K_{i0}>0,K_{ij}=0$ for $j,i\ge1$.
\end{itemize}
\end{itemize}
Literature on convergence to equilibrium (partial interactions): 
\begin{itemize}
	\item 	\cite[Section 5]{BDiPS_2010}: non-rigorous sketch proof that assumes monotonicity of the entropy.\footnote{The argument can be made rigorous by following the ideas in Section~\ref{sec:ex}, and in particular the proof of Theorem~\ref{thm:conv.c}.}
	\item  \cite[Theorem 4]{ZJ_2017}: attempt to rely on decoupling to deduce an exponential convergence rate. (However, the crucial Gronwall type argument therein appears to be false and fixing it would require a new idea.) 
\end{itemize}
Literature on uniqueness (partial interactions): 
\begin{itemize}
	\item \cite[Theorem 5]{ZJ_2017}: uniqueness of weak solutions in a situation where the problem essentially decouples ($K_{0i}=a_0>0$ for all $i\in\{1,\dots,n\}$ and $K_{ij}=0$ for $i,j\ge 1$).
\end{itemize}
Let us finally note that in perturbative or small data settings classical techniques apply. Such regimes are mathematically better behaved and easier to handle:
\begin{itemize}
	\item \cite[Section 3]{BDiPS_2010}: wellposedness of strong solutions near equilibrium.
	\item 	\cite{berendsen2020}:  existence and uniqueness of strong solutions when the interaction coefficients $K_{ij}, i\not=j$ are positive and close to each other. In this case, the system is close to decoupled linear diffusion, and the authors apply classical $L^2$ methods to study the wellposedness in this perturbative setting.
\end{itemize} 

\subsection{New aspects of the present results}
As discussed in the previous paragraph, so far, uniqueness results have only been available for small data settings or under rather restrictive assumptions on the model.
In the present work we remove such smallness restrictions on the model
and establish a weak-strong stability estimate for system~\eqref{eq:sys.all} with general coefficients satisfying~\ref{it:K.sp},~\ref{it:K.pos}, see Theorem~\ref{thm:uniq}.
Our method relies on an adaptation of the technique in~\cite{Hopf_prepr2021} 
devised for proving weak-strong uniqueness in energy-reaction-(cross-)diffusion systems. This approach in turn was motivated by an adjusted relative entropy method due to Fischer~\cite{Fischer_2017} for reaction-diffusion systems.
A crucial ingredient allowing us to handle the new difficulties induced by the size-exclusion effect is a local coercivity-type estimate (cf.\ inequality~\eqref{eq:108}), which we deduce from the strict convexity of the entropy density.

We further establish existence of weak solutions and equilibration in the long-time limit under a fairly mild hypothesis on $K_{ij}$ arguably covering all models relevant for applications.
Here, our idea is to take advantage of the existence result for full interactions~\cite{Juengel_2015}, which provides us with a family of approximate solutions. Upon an analysis of the compactness and convergence properties of these approximate solutions, we then 
take a \enquote{vanishing-interaction limit}, where the artificially introduced interaction parameters are being sent to zero, to eventually obtain weak solutions to the partially interacting system. Owing to poor compactness properties linked to the degeneracy of such systems, we are not able to establish a strong entropy dissipation inequality nor 
monotonicity of the entropy in time. 
However, we will show that a version of the strong entropy dissipation inequality holds true up to an error term (cf.\ inequality \eqref{eq:90.s}) that is initially only known to tend to zero along a \textit{some} sequence of times $t_k\to\infty$, along which stronger convergence properties hold.
Inequality~\eqref{eq:90.s}, however, allows us to upgrade this result and prove the strong convergence to equilibrium along \textit{any} sequence $t\to\infty$.

\subsection{Outline}\label{ssec:outline}

In Section~\ref{sec:prelim}, we use a change of variables to remove the degeneracy of equation~\eqref{eq:sys.all} that is linked to the volume constraint.
	We further establish basic algebraic identities and coercivity estimates involving the mobility matrix that will be needed in subsequent parts.
 The  stability results for fully interacting systems
  are proved in Section~\ref{sec:stab}.
In Section~\ref{sec:ex}, the full-interaction hypothesis is dropped, and general, partially interacting systems relevant for applications are introduced. For these problems, we establish existence of weak solutions and convergence to equilibrium.

This paper has two appendices. Appendix~\ref{app:aux.est} contains a technical lemma used in the proof of the weak-strong stability estimate.
For details concerning  our notations we refer to Appendix~\ref{app:not}.

\section{Preliminaries}\label{sec:prelim}

\subsection{The gradient system for \texorpdfstring{$n$}{} species}\label{sec:entropy}

Consistent with the size constraint, the mobility $\mathbb{M}$ underlying system~\eqref{eq:sys.all} degenerates in the direction $\surfnormal:=(1,\dots,1)^T$ normal to the hyperplane containing $\hypersurf$.
This degeneracy can be removed by means of an affine coordinate transformation mapping the $n$-dimensional domain $\mathcal{D}\subset\mathbb{R}^n$, given by
\begin{align}
	\mathcal{D}:= \left\{ \uvec'=(\uvec_1',\ldots,\uvec_n') \in (0,\infty)^{n}:\;\sum_{i=1}^n \uvec_i'< 1\right\},
\end{align}
diffeomorphically onto $\hypersurf$. Such a transformation can be realised by the map
\begin{align}
	\Psi:\mathcal{D}\to\hypersurf, \qquad U=(U_1,\dots,U_n)\mapsto (1-\sum_{i=1}^nU_i,U_1,\dots,U_n).
\end{align}
By means of this change of variables,  the gradient structure with respect to $\hypersurf$ 
gives rise to a gradient flow for the vector of n densities $U(t,x)\in\mathcal{D}$, as will be detailed below.
Of course, we could have immediately written down a PDE system for $U$ equivalent to~\eqref{eq:sys.all} by replacing $u_0$ by $1-\sum_{i=1}^nU_i$ and rearranging terms.
However, we hope that the following few expository paragraphs 
may help better understand some of the underlying structures.

\paragraph{\em Entropy density}
For $U\in \mathcal{D}$ we define the pulled-back entropy density
\begin{align}
	\hh(\uvec) := h(\Psi(U))= \sum_{i=1}^n \lambda(\uvec_i) +  \lambda\big(1-\sum_{i=1}^n \uvec_i\big),
\end{align}
where as before $\lambda(s)=s\log s-s+1$.
Observe that $\hh\in C(\mathcal{\ol D})\cap C^\infty(\mathcal{D})$ is strictly convex.
Abbreviating $v:=1-\sum_{i=1}^n \uvec_i$, we have 
\begin{align}
	D_i\hh(\uvec) = \log(\uvec_i)-\log(v)=\log\big(\tfrac{\uvec_i}{v}\big)
\end{align}
and 
\begin{align}\label{eq:104}
	D^2\hh(\uvec)=\diag(\tfrac{1}{\uvec_1},\dots,\tfrac{1}{\uvec_n}) + \tfrac{1}{v}E,
\end{align}
where $E_{ij}=1$ for all $i,j\in\{1,\dots,n\}$. 

\paragraph{\em Mobility} Since the Onsager operator $\divv(\mathbb{M}(u)\nabla\,\cdot\,)$ is acting on the dual variables, computing the appropriate mobility $\MM$ on $\mathcal{D}$ requires using the inverse $\Phi:=\Psi^{-1}$  of $\Psi$,
\begin{align}
	\Phi:\hypersurf\to\mathcal{D}, \quad u=(u_0,\dots, u_n)\mapsto U=(u_1,\dots, u_n).
\end{align}
We may now define the  $(n\times n)$ mobility matrix
\begin{align}\label{eq:mobili-pb}
\MM(\uvec):=D\Phi_{|\Psi(U)}\mathbb{M}(\Psi(U))D\Phi^T_{|\Psi(U)},\quad U\in\mathcal{D},
\end{align}
which is again symmetric. (Since $\Phi$ is affine, the argument in $D\Phi_{|\Psi(U)}$ may be dropped.)

\paragraph{\em Gradient-flow equation}

The above ingredients determine the gradient-flow equation~\eqref{eq:sys.gf} in the new variables as
\begin{equation}\label{eq:sys.M}
	\left\{
	\begin{array}{rll}
		\partial_t \uvec - \nabla \cdot (\MM(\uvec) \nabla D\hh(\uvec)) &= 0 & \text{ in }\quad (0,T^*) \times \Omega,\\
		\MM(\uvec) \nabla D\hh(\uvec) \cdot \nu&= 0 & \text{ on }\quad (0,T^*) \times \partial\Omega.\\
	\end{array}
	\right.  
\end{equation}
Introducing the \enquote{diffusion matrix} 
\begin{align}
\hA(\uvec):=	\MM(\uvec)D^2\hh(\uvec),
\end{align}
equation \eqref{eq:sys.M} may further be written as the PDE system
\begin{equation}\label{eq:sys.red}
	\left\{
	\begin{array}{rll}
		\partial_t \uvec - \nabla \cdot (\hA(\uvec) \nabla \uvec) &= 0 & \text{ in }\quad (0,T^*) \times \Omega,\\
		(\hA(\uvec)\nabla \uvec) \cdot \nu&= 0 & \text{ on }\quad (0,T^*) \times \partial\Omega.\\
	\end{array}
	\right.  
\end{equation}
A direct computation yields the following explicit form for the matrix $\hA(U)$ 
\begin{equation}\label{eq:A.red}
	\hA_{ij}(\uvec)=\delta_{ij}\left[\sum_{l=1}^n K_{il}u_l+K_{i0}(1-\sum_{l=1}^n u_l)\right]
	-(K_{ij}-K_{i0})u_i,  \quad i,j\in\{1,\ldots, n\}.
\end{equation}
Observe that system~\eqref{eq:sys.red} can also directly be obtained from~\eqref{eq:sys.all} by inserting $u_0=1-\sum_{i=1}^nU_i$ and rearranging terms.

\paragraph{\em Entropy dissipation}

We emphasize that the information contained in~\eqref{eq:sys.M} is the same as in the original system. 
In particular, we assert that 
\begin{align}\label{eq:dq.matrix}
	D^2\hh\MM D^2\hh = D\Psi^T(D^2h\circ\Psi) (\mathbb{M}\circ\Psi) (D^2h\circ\Psi) D\Psi.
\end{align}
This follows from the identities $D^2\hh=D\Psi^T(D^2h\circ\Psi) D\Psi$ and 
\begin{align}
	(D\Psi D\Phi)\mathbb{M}(u)(D\Psi D\Phi)^T = \mathbb{M}(u),
\end{align}
which can be verified explicitly by an elementary calculation using the structure of $\mathbb{M}$.
Hence, as a consequence of~\eqref{eq:dq.matrix}, the dissipated quantities of the flows in transformed and original coordinates are indeed the same:
\begin{align}\label{eq:146}
	\begin{multlined}
		\PP(U):=\sum_{i,j=1}^n\nabla U_i\cdot (D^2\hh(\uvec)\MM(\uvec)D^2\hh(\uvec))_{ij}\nabla U_j
		\\\hspace{.2\linewidth}= \sum_{k,l=0}^n\nabla u_k\cdot (D^2h(u)\mathbb{M}(u)D^2h(u))_{kl}\nabla u_l =:\scrp(u).
	\end{multlined}
\end{align}

\paragraph{\em Parabolic structure}

Assuming a mild non-degeneracy condition on the coefficients $K_{ij}$,
 the problem formulated in terms of $n$ species enjoys full parabolicity in a classical sense, as will be explained in the following.
For simplicity, we only consider the case of a fully interacting system obeying~\ref{it:K.pos}, which guarantees that
$\kappa := \min_{i,j\in\{0,\dots,n\},i\not=j}K_{ij}$ is strictly positive.
 It has been shown in the proof of~\cite[Lemma~2.3]{BE_2018} that 
\begin{align}\label{eq:102}
	\zeta^T D^2\hh(\uvec)\hA(\uvec) \zeta \ge \kappa \sum_{i=1}^n \uvec_i^{-1}\zeta_i^2\qquad\text{ for } U\in\mathcal{D}.
\end{align}
See also Corollary~\ref{cor:hM.coerc} below, which actually implies this  bound (although our reasoning is different and perhaps easier to grasp).
In terms of $\MM$, inequality~\eqref{eq:102} takes the form
\begin{align}\label{eq:102'}
	\zeta^T D^2\hh(\uvec)\MM(\uvec)D^2\hh(\uvec) \zeta \ge  \kappa\sum_{i=1}^n \uvec_i^{-1}\zeta_i^2 \qquad \text{ for } U\in\mathcal{D}.
\end{align}
Since for each $U\in \mathcal{D}$ the matrix $D^2\hh(U)$ is invertible, the last inequality
implies that $\MM(\uvec)>0$ for all $U\in\mathcal{D}$, where we recall that $\MM(\uvec)$ is symmetric. Hence, by continuity,
for every compact subset $\mathcal{K}\subset\subset \mathcal{D}$
there 
exists a positive constant 
$c_*(\mathcal{K})>0$ such that 
\begin{align}\label{eq:M.nd} 
	\MM(\uvec)\ge c_*(\mathcal{K})\diag(1,\dots,1)\qquad\text{ for all }U\in\mathcal{K}.
\end{align}
We further note that, since for every $\uvec\in \mathcal{D}$ the diffusion matrix $\hA(\uvec)$ is the product of the positive definite symmetric matrices $\MM(\uvec)$ and $D^2\hh(\uvec)$, it has its spectrum $\sigma(\hA(\uvec))$ contained in $\mathbb{R}_{>0}$.
Therefore, thanks to Amann's work~\cite{Amann_1990_II}, which requires the real parts of the eigenvectors of $\hA(\uvec)$ to be strictly positive, the Cauchy problem associated with system~\eqref{eq:sys.red} has a unique local-in-time smooth solution for Lipschitz regular initial data $\Uin$ satisfying $\Uin(\ol\Om)\subset\mathcal{D}$.

\begin{rem}
\normalfont In fact, under hypothesis~\ref{it:K.pos}, the inclusion $\sigma(\hA(\uvec))\subset\mathbb{R}_{>0}$ sufficient for applying~\cite{Amann_1990_II} 
even holds for all $\uvec\in \mathcal{\ol D}$ 
(and thus by continuity also in an open neighbourhood of $\mathcal{\ol D}$).
This is a consequence of~\cite[Lemma~2.3]{JS_2013}.
 Hence, local wellposedness is even valid under the more general condition $\Uin(\ol\Om)\subset\mathcal{\ol D}$.
Preservation of positivity and size constraint are, however, not fully trivial if  $\Uin(\ol\Om)\cap\partial\mathcal{D}\not=\emptyset$ and will have to be verified a posteriori. 
\end{rem}

\paragraph{\em Equivalence of the two PDE systems}
The $n$-species diffusion system~\eqref{eq:sys.red} precisely agrees with the last $n$ components of the original system~\eqref{eq:sys.all}. This may be verified explicitly by using formula~\eqref{eq:A.red} and the relation $\sum_{i=0}^nu_i=1$.
Since the equation for the zeroth component $u_0=1-\sum_{i=1}^nU_i$ in~\eqref{eq:sys.all} is obtained by summing up the equations for the components $i=1,\dots, n$, system~\eqref{eq:sys.red} is equivalent to the original problem. For later reference, we formulate this observation for the concept of weak solutions in the following lemma.
\begin{lemma}\label{f:12}
	System~\eqref{eq:weak} for $i\in\{0,\dots, n\}$ is  equivalent to the weak form of~\eqref{eq:sys.red}, i.e.\ to
	\begin{align}\label{eq:weak.red}
		\int_0^T\langle\partial_tU_i,\psi\rangle\,\dd t 
		+	\int_0^T\!\int_\Om \sum_{j=1}^n\hat A_{ij}(U)\nabla U_j\cdot\nabla\psi\,\dd x\dd t =0
	\end{align}
for $i\in\{1,\dots, n\}$.
	More precisely, $u=(u_0,\dots,u_n)$ being a weak solution of  \eqref{eq:sys.all} in the sense of Definition \ref{def:weak.sol} is equivalent to $U=(u_1,\dots, u_n)$ satisfying
	$u_i\in L^2(0,T;H^1(\Om))$, $\partial_tu_i\in L^2(0,T;(H^1(\Om))^*)$ for all $T<T^*$ and $i\in\{1,\dots,n\}$, fulfilling $U(t,x)\in \mathcal{\ol D}$ for a.e.\ $(t,x)\in(0,T^*)\times\Om$, and obeying~\eqref{eq:weak.red} for all $T<T^*$ and all $\psi\in L^2(0,T;H^1(\Om))$.
\end{lemma}

We should note that the dimension reduction leading to full parabolicity comes at a price: the entropy function in the new coordinates induces a strong coupling between the $n$ unknowns, potentially introducing new difficulties in the uniqueness and stability analysis.

\subsection{Algebraic estimates}\label{ssec:matrix}
Here, we establish coercivity bounds for the quadratic form induced by the symmetric and positive semi-definite matrix
\begin{align}\label{eq:140}
	\mathbb{P}(u) := D^2h(u)\mathbb{M}(u)D^2h(u), \quad u\in \hypersurf,
\end{align}
on the tangent space $T_u\hypersurf$ to $\hypersurf$.
Since the vector $\surfnormal:=(1,\dots,1)$ points in the direction perpendicular to the hyperplane containing $\hypersurf$, we can identify
\begin{align}
	T_u\hypersurf=\{\xi\in \mathbb{R}^{1+n}:\xi\cdot\surfnormal=0\}.
\end{align}

\begin{lemma}\label{l:EPalg}
	Let $K_{ij}$ satisfy Hypotheses~\ref{it:K.sp} \&~\ref{it:K.nn}.
	Further recall that the mobility matrix $\mathbb{M}$ is given by~\eqref{eq:mobility}, and let $\mathbb{P}$ be  given by~\eqref{eq:140} (with $h$ as in~\eqref{eq:def.H}).
	For all $u\in \hypersurf$ and all $\xi\in T_u\hypersurf$ 
	\begin{align}\label{eq:141}
		\xi^T	\mathbb{P}(u) \xi = \frac{1}{2}\sum_{i, j}K_{ij}u_iu_j\left|\frac{\xi_i}{u_i}-\frac{\xi_j}{u_j}\right|^2.
	\end{align}
\end{lemma}
While identity~\eqref{eq:141} holds for general $\xi\in\mathbb{R}^{1+n}$, only tangent vectors~$\xi$ are relevant in our analysis.
 
\begin{proof}[Proof of Lemma~\ref{l:EPalg}]
	Observe that $D^2h(u)_{ij}=\frac{1}{u_i}\delta_{ij}$, and hence
	\begin{align}
		\mathbb{P}_{ij}(u) = \frac{1}{u_j}\delta_{ij}\sum_{k=0}^nK_{ik}u_k-K_{ij}.
	\end{align}
We then compute
\begin{align}
	\sum_{i,j}	\mathbb{P}_{ij}(u)\xi_i\xi_j &= 	\sum_{i,k}K_{ik}\frac{u_k}{u_i}\xi_i^2
	-	\sum_{i,j}K_{ij}\xi_i\xi_j 
	\\&=\frac{1}{2}\sum_{i,j}K_{ij}u_iu_j\big(\left|\tfrac{\xi_i}{u_i}\right|^2+\left|\tfrac{\xi_j}{u_j}\right|^2\big)-	\frac{1}{2}\sum_{i,j}K_{ij}u_iu_j2\tfrac{\xi_i}{u_i}\tfrac{\xi_j}{u_j},
\end{align}
where second step uses the symmetry of $K_{ij}$. The last line equals the right-hand side  of~\eqref{eq:141} and the assertion follows.
\end{proof}

\begin{lemma}[Matrix coercivity estimates]\label{l:M.coerc}
Assume the hypotheses and use the notations of Lemma~\ref{l:EPalg}.
Then for all $u\in \hypersurf$ and $\xi\in T_u\hypersurf$ 
\begin{align}\label{eq:pdGrad}
	\xi^T	\mathbb{P}(u) \xi \ge \sum_{\alpha=0}^n\kappa^{(\alpha)}
	\bigg(\frac{1}{u_\alpha}|\xi_\alpha|^2+u_\alpha\sum_{j=0}^n\frac{1}{u_j}|\xi_j|^2\bigg),
\end{align}
where $\kappa^{(\alpha)}:=\frac{1}{2}\min_{\ell:\ell\not=\alpha}K_{\alpha\ell}\ge0$ for $\alpha\in\{0,\dots,n\}$.

In particular, 
\begin{align}\label{eq:ffGrad}
		\xi^T	\mathbb{P}(u) \xi \ge\min_{\alpha, \ell:\alpha\not=\ell}K_{\alpha\ell}\cdot\sum_{j=0}^n\frac{1}{u_j}|\xi_j|^2.
\end{align}
\end{lemma}

\begin{proof}
	Identity~\eqref{eq:141} allows us to estimate
	\begin{align}
			\xi^T	\mathbb{P}(u) \xi &= \frac{1}{2}\sum_{\alpha=0}^n\sum_{j:j\not=\alpha}K_{\alpha j}u_\alpha u_j\left|\frac{\xi_\alpha}{u_\alpha}-\frac{\xi_j}{u_j}\right|^2
			\\&\ge\sum_{\alpha=0}^n\kappa^{(\alpha)} u_\alpha
			\sum_{j:j\not=\alpha} u_j
			\bigg(\frac{1}{u_\alpha^2}\xi_\alpha^2+\frac{1}{u_j^2}\xi_j^2 - 2\frac{1}{u_\alpha u_j}\xi_\alpha\xi_j	\bigg)
\\&=\sum_{\alpha=0}^n\kappa^{(\alpha)} u_\alpha
\bigg(\frac{1-u_\alpha}{u_\alpha^2}\xi_\alpha^2+\sum_{j:j\not=\alpha} \frac{1}{u_j}\xi_j^2 + 2\frac{1}{u_\alpha}\xi_\alpha^2	\bigg)
\\&=\sum_{\alpha=0}^n\kappa^{(\alpha)}
\bigg(\frac{1}{u_\alpha}|\xi_\alpha|^2+u_\alpha\sum_{j=0}^n\frac{1}{u_j}|\xi_j|^2\bigg),
	\end{align}
	where the penultimate equality uses the fact that $\xi\cdot \surfnormal=0$, that is, $\sum_{j=0}^n\xi_j=0$. This shows~\eqref{eq:pdGrad}.
	
	 To infer inequality~\eqref{eq:ffGrad} it suffices to note that the right-hand side of~\eqref{eq:pdGrad} is bounded below by
	\begin{align}
	 \frac{1}{2}\min_{\alpha, \ell:\alpha\not=\ell}K_{\alpha\ell}\cdot\sum_{\beta=0}^n
		\bigg(\frac{1}{u_\beta}|\xi_\beta|^2+u_\beta\sum_{j=0}^n\frac{1}{u_j}|\xi_j|^2\bigg)
		= \frac{1}{2}\min_{\alpha, \ell:\alpha\not=\ell}K_{\alpha\ell}\cdot\bigg(2\sum_{j=0}^n\frac{1}{u_j}|\xi_j|^2\bigg).
	\end{align}
\end{proof}

\begin{corollary}\label{cor:hM.coerc} 
	Let $K_{ij}$ satisfy Hypotheses~\ref{it:K.sp}, \ref{it:K.nn} and define
	\begin{align}\label{eq:140hat}
		\mathbb{\hat P}(U) := D^2\hh(U)\MM(U) D^2\hh(U),\qquad U\in \mathcal{D}.
	\end{align}
Then, for all $U\in \mathcal{D}$ and $\zeta\in \mathbb{R}^n$,
\begin{align}
	\zeta^T	\mathbb{\hat P}(U)\zeta \ge 
	\kappa\bigg(|\sum_{l=1}^n\zeta_l|^2
	+\sum_{i=1}^n\frac{1}{U_i}|\zeta_i|^2\bigg),
\end{align}
where $\kappa := \min_{j,k\in\{0,\dots,n\},j\not=k}K_{jk}$.
\end{corollary}

\begin{proof}
This follows from Lemma~\ref{l:M.coerc} by inserting the tangent vector  $\xi:=D\Psi\zeta$ and performing some elementary manipulations. Note that here we have also used the  identity~\eqref{eq:dq.matrix}.	
\end{proof}

\subsection{Existence of weak solutions under 
	hypothesis~\texorpdfstring{\ref{it:K.pos}}{}}

Observe that under hypothesis~\ref{it:K.pos}, i.e.\ for $\kappa:=\min_{i\not=j}K_{ij}>0$,
Corollary~\ref{cor:hM.coerc} implies that $
	\zeta^T	\mathbb{\hat P}(U)\zeta \gtrsim \sum_{i=1}^n\frac{1}{U_i}|\zeta_i|^2$.
This estimate, combined with the convexity properties of $\hh$, allows us to apply 
Theorem~2 of~\cite{Juengel_2015} to infer 
 global existence of weak solutions to system~\eqref{eq:sys.all} under the condition~\ref{it:K.pos}, where we also rely on Lemma~\ref{f:12} on the equivalence between the original and the transformed system:
 
\begin{theorem}[See~\cite{Juengel_2015} \& L.~\ref{f:12}]\label{thm:ex.full}
	Let $\min_{i\not=j}K_{ij}>0$.	Let $\uin\in L^1(\Om)^\sn$ satisfy $\uin(x)\in\hsurfCl$ for a.e.\ $x\in\Om$.
	Then there exists a function $u\in L^2_\loc([0,\infty);H^1(\Om))^\sn$ with 
	$\partial_tu\in L^2_\loc([0,\infty);H^1(\Om)^*)^\sn$ and $u(t,x)\in\hsurfCl$ for a.e.\ $(t,x)\in(0,\infty)\times\Om$ that satisfies system~\eqref{eq:sys.all} in the weak sense and takes the initial datum $u(0)=\uin$. 	
	Furthermore, the regularity
	 \begin{align}\label{eq:sqrtL2H1}
		\sqrt{u_i}\in L^2_\loc([0,\infty);H^1(\Om))\qquad \textit{ for all }i\in\{0,\dots,n\}
	\end{align}
holds true.
\end{theorem}
The main feature of the proof of Theorem~\ref{thm:ex.full} consists in a transformation to the Legendre conjugate variables $w=D\hh(U)$ upon which the system can be regularised by a standard higher-order elliptic term while preserving the key entropy estimate.

We should note that the statement of Theorem~2 by J\"ungel~\cite{Juengel_2015} 
assumes that $\uin(x)\in\hypersurf$ for a.e.\ $x\in\Om$. However, as observed by the author in the paragraph following the statement of this theorem, an approximation argument allows to relax this hypothesis to $\uin(x)\in\hsurfCl$ for a.e.\ $x\in\Om$. 
It is further necessary to point out that the regularity~\eqref{eq:sqrtL2H1} does not appear in the statement of~\cite[Theorem~2]{Juengel_2015}, but it directly follows from the proof of this result, see~\cite[proof of Theorem~2;  in particular, equation (23) and Step~3]{Juengel_2015}.

 Let us remark that in~\cite[Theorem~2.2]{BE_2018} a result similar to Theorem~\ref{thm:ex.full} appears (formulated for the $n$-species system), again deduced by verifying the hypotheses of Theorem~2 of~\cite{Juengel_2015}.
 But since it is stated in a slightly weaker form, we will not use it here.

\section{Stability results for full interactions}\label{sec:stab}

\subsection{Entropy dissipation balance}

Theorems~\ref{thm:uniq} and~\ref{thm:stab.ss} both rely on the following entropy dissipation identity.

\begin{proposition}\label{prop:edi}
	Let $u=(u_0,\dots,u_n)$ be a weak solution according to Definition~\ref{def:weak.sol} enjoying the regularity $\sqrt{u_i}\in L^2_\loc([0,T^*);H^1(\Om))$ for all $i$.
	Then for all $T\in(0,T^*)$
	\begin{align}\label{eq:edi}
		\int_\Om h(u(T))\,\dd x =\int_\Om h(u(0))\,\dd x -\int_0^T\!\int_\Om \scrp(u)\,\dd x\dd t,
	\end{align}
	where $\scrp(u)=\sum_{i,j}\nabla u_i\cdot\mathbb{P}_{ij}(u)\nabla u_j=\sum_{i,j}\nabla u_i\cdot(D^2h(u)\mathbb{M}(u)D^2h(u))_{ij}\nabla u_j$ (cf.~\eqref{eq:146}).
	
	As a consequence, for all $0\le s<T<T^*$
	\begin{align}\label{eq:edb}
		\int_\Om h(u(T))\,\dd x =\int_\Om h(u(s))\,\dd x -\int_s^T\!\int_\Om \scrp(u)\,\dd x\dd t.
	\end{align}
\end{proposition}

\begin{proof}[Proof of Prop.~\ref{prop:edi}]
	We begin by observing that  $u\in C([0,T];L^2(\Om))^{1+n}$. This is a consequence of the regularity hypotheses in Definition \ref{def:weak.sol} combined with standard embeddings for Bochner spaces. Formally, identity~\eqref{eq:edi} follows upon testing~\eqref{eq:weak} with $D_ih(u)$, taking the sum over $i$ and applying a chain rule.
	However, the Boltzmann-type entropy density $h$ is not smooth up to $u_i=0$, and hence, we need to regularise first. The procedure is very standard, but since it has not been carried out for this particular system, some details are required.
	
	For $\ve>0$ let $h_\ve(u):=h(u_0+\ve,\dots,u_n+\ve)=\sum_{i=0}^n\lambda(u_i+\ve).$
	Then
	\begin{align}
		D_ih_\ve(u)=\log(u_i+\ve)\quad\in \,L^2_\loc([0,T^*);H^1(\Om))
	\end{align}
	and $D_{ik}h_\ve(u) = \tfrac{1}{u_i+\ve}\delta_{ik}$ for all $i,k\in\{0,\dots,n\}$,
	where $\delta_{ik}$ denotes the Kronecker symbol.
	
	Choosing $\psi=D_ih_\ve(u)$ in the weak equation~\eqref{eq:weak} for $u_i$ and taking the sum $\sum_{i=0}^n$ gives
	\begin{align}\label{eq:114}
		\int_\Om h_\ve(u(T))\,\dd x - 	\int_\Om h_\ve(u(0))\,\dd x 
		=	-\int_0^T\!\int_\Om \sum_{i,j,k}D_{ik}h_\ve(u)A_{ij}(u)\nabla u_j\cdot \nabla u_k\,\dd x\dd t. \qquad
	\end{align}
	To pass to the limit $\ve\to0$, we wish to apply the dominated convergence theorem. 
	In the terms on the left-hand side this is immediate thanks to the bound  $0\le u_i\le 1$.
	Concerning the integral on the right-hand side, we need appropriate information about its integrand. 
	For this purpose, we compute for $i,j\in\{0,\dots,n\}$, using the formula~\eqref{eq:mobility} for $\mathbb{M}$,
	\begin{align}
		 \sum_{k}D_{ik}h_\ve(u)A_{ij}(u)\nabla u_j\cdot \nabla u_k &= 
		 \frac{1}{u_i+\ve}\mathbb{M}_{ij}(u)\frac{1}{u_j}\nabla u_j\cdot\nabla u_i
		 \\&=\frac{1}{u_i+\ve} u_i(\delta_{ij}\sum_{l=0}^nK_{il}u_l - K_{ij}u_j)\frac{1}{u_j}\nabla u_j\cdot\nabla u_i
		 \\&=\delta_{ij}\frac{1}{u_i+\ve}|\nabla u_i|^2\sum_{l=0}^nK_{il}u_l
		 -\frac{u_i}{u_i+\ve} K_{ij}\nabla u_j\cdot\nabla u_i.
	\end{align}
The terms in the last line are bounded above (in absolute values, pointwise a.e.) by the integrable function $C\sum_i|\nabla \sqrt{u_i}|^2$ for some $\ve$-independent constant $C\in(0,\infty)$.
	Hence, the dominated convergence theorem allows to infer identity~\eqref{eq:edi} upon taking the limit $\ve\downarrow0$ in equation \eqref{eq:114}.
\end{proof}

\subsection{Weak-strong stability estimate}\label{ssec:stab.est}

In the derivation of the stability estimate asserted in Theorem~\ref{thm:uniq}, we roughly follow the technique in~\cite{Hopf_prepr2021}. The size-exclusion effect, however, induces some new difficulties requiring the new estimate~\eqref{eq:108} to handle the present problem.

\begin{proof}[Proof of Theorem~\ref{thm:uniq}]
Let $\tilde u=(\tilde u_0,\dots,\tilde u_n)$ be a strong solution and $u=(u_0,\dots,u_n)$ a weak solution with the properties as stated in Theorem~\ref{thm:uniq}. 
We recall the relation $U=(u_1,\dots, u_n)$, $\tilde U=(\tilde u_1,\dots,\tilde u_n)$ and 
note that for $i\in\{1,\dots,n\}$
\begin{align}\label{eq:112}
	D_i \hh(U)=\log\bigg(\frac{U_i}{u_0}\bigg)=\log\big(\tfrac{u_i}{u_0}\big).
\end{align}
In the proof below, we freely switch between the original and the transformed variables $u,\tilde u$ and  $U,\tilde U$  and choose the form that appears to be more convenient for the specific argument in question. As explained in Section~\ref{sec:entropy}, both forms are equivalent and the specific choice of variables has no mathematical significance.

We then define 
\begin{align}
	\hh_\rel(U,\tilde U) &:= \hh(U)-\sum_{i=1}^nD_i \hh(\tilde U)(U_i-\tilde U_i)-\hh(\tilde U)
	\\&=h(u)-\sum_{k=0}^nD_k h(\tilde u)(u_k-\tilde u_k)-h(\tilde u)=h_\rel(u,\tilde u).
\end{align}
Since $\hh$ is strictly convex uniformly in $\mathcal{D}$ and $u_0=1-\sum_{i=1}^nu_i$, we have  the lower bound
\begin{align}
	\hh_\rel(U,\tilde U) \gtrsim |U-\tilde U|^2\sim |u-\tilde u|^2.
\end{align}

Let now $T\in (0,T^*)$.
	We will establish a weak-strong stability estimate of the form
	\begin{align}\label{eq:115}
		\eval{\int_\Om \hh_\rel(U,\tilde U)\,\dd x}_{\tau=0}^{\tau=t} &\le C_{T,\tilde u} 
		 \int_0^t\!\!\int_\Om \;\hh_\rel(U,\tilde U)\;\dd x\dd \tau,\quad t\in(0,T),
	\end{align}
with $C_{T,\tilde u}=C(\|\tilde u\|_{C^{0,1}([0,T]\times\ol\Om)},(\inf_{[0,T]\times\ol\Om} \tilde u_i)_{i=0}^n)$,
 upon which the asserted inequality~\eqref{eq:stab.f} follows from Gronwall's lemma.

To begin with, we compute the time evolution of $\hh_\rel(U,\tilde U)$ using the entropy dissipation balance~\eqref{eq:edb} and the regularity of $U,\tilde U$ as well as integration by parts for Sobolev--Bochner functions (cf.\ \cite[Lemma~7.3]{Roubicek_2013_npde}),
\begin{align}
	\eval{\int_\Om \hh_\rel(U,\tilde U)\,\dd x}_{\tau=0}^{\tau=t} 
	&=-\int_0^t\!\!\int_\Om \PP(U) \,\dd x\dd \tau
		-\int_0^t \frac{\dd}{\dd\tau} \!\int_\Om \sum_{i=1}^nD_i \hh(\tilde U)(U_i-\tilde U_i)+\hh(\tilde U)\, \dd x\,\dd \tau
		\\&=-\int_0^t\!\!\int_\Om \PP(U) \,\dd x\dd \tau
		-\int_0^t\!\!\int_\Om\big(\sum_{i,k=1}^n D_{ik} \hh(\tilde U)\partial_t\tilde U_k(U_i-\tilde U_i)\big) \,\dd x\dd \tau
		\\&\hspace{0.25\linewidth} -\int_0^t\sum_{i=1}^n \langle\partial_tU_i,D_i \hh(\tilde U)\rangle\,\dd \tau.
\end{align}
For the penultimate integral on the right-hand side we next use the fact that $\tilde u$ is a strong solution, while the last integral on the right-hand side is further rewritten appealing to the weak solution property~\eqref{eq:weak.red} of $U$. Upon insertion of the identity $\hA= \MM D^2\hh$, some rearrangement of terms and renaming of indices, we arrive at
\begin{equation}\label{eq:110}
	\begin{split}
		\quad\eval{\int_\Om \hh_\rel(U,\tilde U)\,\dd x}_{\tau=0}^{\tau=t} &=
		\int_0^t\!\!\int_\Om \;\rho\;\dd x\dd \tau,\quad t\in(0,T),
	\end{split}
\end{equation}
where\footnote{A caveat is in place here: Since the function $D_l\hh(U)$ may not be locally integrable (not to mention weakly differentiable in the Sobolev sense), the gradient $\nabla D_l\hh(U)$ in~\eqref{eq:109} needs to be read in a symbolic way with the rigorous meaning $\nabla D_l\hh(U):=\sum_{j=1}^nD_jD_l\hh(U)\nabla U_j$.} 
\begin{equation}\label{eq:109}
	\begin{split}
		\rho :=-\PP(U)
		&+ \sum_{i,l=1}^n\nabla D_i \hh(\tilde U)\cdot \MM_{il}(U)\nabla D_l \hh(U)
		\\& + \sum_{i,l,j=1}^n\nabla\big(D_{ij}\hh(\tilde U)(U_j-\tilde U_j)\big)\cdot 
		\MM_{il}(\tilde U)\nabla D_l \hh(\tilde U)
	\end{split}
\end{equation}
with the understanding that $\nabla D_l\hh(U):=\sum_{k=1}^nD_{lk}\hh(U)\nabla U_k$.

Hence, for proving the stability estimate~\eqref{eq:115}
it suffices to show the pointwise bound 
\begin{align}\label{eq:111}
		\rho\le C_{T,\tilde u}\, \hh_\rel(U,\tilde U).
\end{align}
For this purpose, we fix $\iota:=\iota_T\in(0,1)$ such that $\tilde u_i\ge 2\iota$ in $[0,T]\times\Om$ for all $i\in\{0,\dots,n\}$. 
For a.e.\ $(t,x)\in [0,T]\times\Om$ we then distinguish the following cases:
\begin{itemize}
	\item Case 1: $\min\{u_0(t,x),\dots,u_n(t,x)\}> \iota$.
	\item Case 2:  $\min\{u_0(t,x),\dots,u_n(t,x)\}\le \iota$.
\end{itemize}

\underline{Case 1.} In this case, $u_i\ge\iota$ for all $i\in\{0,\dots,n\}$ and there exists a convex set $\mathcal{K}_\iota\subset\subset\mathcal{D}$ 
with $\dist(\mathcal{K}_\iota,\partial\mathcal{D})\gtrsim\iota$  such that $U(t,x)\in \mathcal{K}_\iota$.
We will show that $\rho\le C_{T,\tilde u}|u-\tilde u|^2$ using the nondegeneracy property~\eqref{eq:M.nd}, the strict convexity of~$\hh$, as well as the smoothness of $\MM$ and $\hh$ in $\mathcal{D}$. 
Dependencies of constants on fixed quantities such as $\|\tilde u\|_{C^{0,1}([0,T]\times\ol\Om)}$ and $\iota$ will frequently be omitted.

Recalling the identity $\PP(U)=\sum_{i,l=1}^n\nabla D_i \hh(U)\cdot\MM_{il}(U)\nabla D_l \hh(U)$ and taking advantage of the fact that, since $U(t,x)$ is away from $\partial\mathcal{D}$,  the Hessian $D^2\hh(U(t,x))\in \mathbb{R}^{n\times n}$ is well-defined, we may rewrite~\eqref{eq:109} in the form 
\begin{align}\label{eq:103}
	\begin{aligned}
		\rho & =-\sum_{i,l=1}^n\nabla (D_i \hh(U)-D_i \hh(\tilde U))\cdot\MM_{il}(U)\nabla (D_l \hh(U)-D_l \hh(\tilde U))
		\\&\qquad -\sum_{i,l=1}^n\nabla (D_i \hh(U)-D_i \hh(\tilde U))\cdot(\MM_{il}(U)-\MM_{il}(\tilde U))\nabla D_l \hh(\tilde U)
		\\&\qquad -\sum_{i,j,l=1}^n \nabla\big(D_i \hh(U)-D_i \hh(\tilde U)-D_{ij}\hh(\tilde U)(U_j-\tilde U_j)\big)\cdot 
		\MM_{il}(\tilde U)\nabla D_l \hh(\tilde U),
	\end{aligned}
\end{align}
where we continue to employ the short-hand  notation $\nabla D_i \hh(U):=\sum_{k=1}^nD_{ik} \hh(U)\nabla U_k$.

Thanks to inequality~\eqref{eq:M.nd} with $\mathcal{K}=\mathcal{K}_\iota$ and the fact that $h\in C^4(\mathcal{K}_\iota)$, there exists $\epsilon_*=\epsilon_*(\mathcal{K}_\iota)>0$ such that for any $\delta\in(0,1)$
\begin{align}
	\rho&\le -\epsilon_*\sum_{i=1}^n|\nabla D_i \hh(U)-	\nabla D_i \hh(\tilde U)|^2
+\delta|\nabla u-\nabla \tilde u|^2
	+C_\delta|u-\tilde u|^2.
\end{align}
Here, the last term in~\eqref{eq:103} has been estimated invoking Taylor's theorem, see Appendix~\ref{app:aux.est} for details.

We assert that there exists $\epsilon(\iota_T)>0$ and $C_{T,\tilde u}=C(\|\tilde u\|_{C^{0,1}([0,T]\times\ol\Om)},\iota_T)<\infty$ 
such that
\begin{equation}\label{eq:108}
G:=	\sum_{i=1}^n|\nabla D_i \hh(U)-	\nabla D_i \hh(\tilde U)|^2
	\ge \epsilon(\iota_T)|\nabla u-\nabla \tilde u|^2 - C_{T,\tilde u}|u-\tilde u|^2.
\end{equation}
First suppose that~\eqref{eq:108} holds true. Then, choosing $\delta=\epsilon_*\epsilon(\iota_T)$  allows us to deduce~\eqref{eq:111} in Case~1.

\smallskip

\textit{Proof of inequality}~\eqref{eq:108} (in Case~1). 
We compute for $i\in\{1,\dots,n\}$
\begin{align}
	\nabla D_i\hh(U)-	\nabla D_i \hh(\tilde U) &= 
	\sum_{j=1}^n\big(D_{ij}\hh(U)\nabla U_j-D_{ij}\hh(\tilde U)\nabla \tilde U_j\big)
	\\&=	\sum_{j=1}^n \big[D_{ij}\hh(U)(\nabla U_j-\nabla \tilde U_j)
	+(D_{ij}\hh(U)-D_{ij}\hh(\tilde U))\nabla \tilde U_j\big].
\end{align}
Since $\hh$ is smooth and strictly convex on $\mathcal{D}$, the inverse Hessian $B(U):=(D^2\hh(U))^{-1}$
is well-defined and for $U\in \mathcal{K}_\iota$ its spectral norm $|||B(U)|||$ is bounded above by a finite constant $C(\iota)$ only depending on $\iota>0$. Hence, 
\begin{align}
	|\nabla U-\nabla \tilde U|&\le |B(U)(\nabla D\hh(U)-	\nabla D \hh(\tilde U) )|+
	|B(U)(D^2\hh(U)-D^2\hh(\tilde U))\nabla \tilde U|
	\\&\le C(\iota)|\nabla D\hh(U)-	\nabla D \hh(\tilde U)|+C_{T,\tilde u}|U-\tilde U|,
\end{align}
which implies~\eqref{eq:108}.

This concludes the proof of inequality~\eqref{eq:108}.  

\medskip

\underline{Case 2.} 
In this case, we recall the definition of $\rho$ in~\eqref{eq:109} and Corollary~\ref{cor:hM.coerc} (see also inequality~\eqref{eq:102'}) to estimate
\begin{align}
	\rho &\le -\epsilon_1\sum_{i=1}^n|\nabla\sqrt{u_i}|^2+C_{T,\tilde u}|\nabla u|+C_{T,\tilde u}
	\\&\le -\epsilon_2\sum_{i=1}^n|\nabla\sqrt{u_i}|^2+C_{T,\tilde u},
\end{align}
where we have used the boundedness of $u$.
Since $|u-\tilde u|\gtrsim_\iota1$, we deduce~\eqref{eq:111}.

\smallskip
The proof of Theorem~\ref{thm:uniq} is now complete.
\end{proof}

\begin{remark}[Second proof of inequality~\eqref{eq:108}]
	For a somewhat more quantitative estimate in Case~1,  the following alternative proof of inequality~\eqref{eq:108} may be of interest.
 Let $T\in(0,\infty)$ be fixed but arbitrary and abbreviate $\iota:=\iota_T$.
	In view of~\eqref{eq:112}, we have
	\begin{align}
		G=\sum_{i=1}^n |\nabla\log(\tfrac{u_i}{u_0})-\nabla\log(\tfrac{\tilde u_i}{\tilde u_0})|^2,
	\end{align}
	which can be rewritten as
	\begin{align}
		G=\sum_{i=1}^n|\nabla\vp_i-\nabla\vp_0|^2
	\end{align}
	for  $\vp_i:=\log\big(\tfrac{u_i}{\tilde u_i}\big)$. We compute 
	using the identity $u_0=1-\sum_{i=1}^nu_i$
	\begin{align}
		\nabla\varphi_0 = \nabla \log(u_0)-\nabla\log(\tilde u_0)
		&= -\sum_{i=1}^n\tfrac{1}{u_0}\nabla u_i
		+\sum_{i=1}^n\tfrac{1}{\tilde u_0}\nabla \tilde u_i
		\\&= -\sum_{i=1}^n\tfrac{u_i}{u_0}\nabla \log(u_i)
		+\sum_{i=1}^n\tfrac{\tilde u_i}{\tilde u_0}\nabla \log(\tilde u_i)
		\\&=-\sum_{i=1}^n\tfrac{u_i}{u_0}\nabla\varphi_i+R(u,\tilde u),
	\end{align}
	where here and below $R(u,\tilde u)$ denotes a harmless term that satisfies
	$|R(u,\tilde u)|\le C_{T,\tilde u}|u-\tilde u|$ and  may change from line to line.
	Hence,
	\begin{align}
		\nabla\vp_1-\nabla\vp_0 &= (1+\tfrac{u_1}{u_0})\nabla\vp_1 +\sum_{i=2}^n\tfrac{u_i}{u_0}\nabla\vp_i +R(u,\tilde u)
		\\&= (1+\sum_{i=1}^n\tfrac{u_i}{u_0})\nabla\vp_1 +\sum_{i=2}^n\tfrac{u_i}{u_0}(\nabla\vp_i-\nabla\vp_1) +R(u,\tilde u)
		\\&=(1+\sum_{i=1}^n\tfrac{u_i}{u_0})\nabla\vp_1 +\sum_{i=2}^n\tfrac{u_i}{u_0}(\nabla\vp_i-\nabla\vp_0+\nabla\vp_0-\nabla\vp_1) +R(u,\tilde u).
	\end{align}
	Using the inequality $2|a-b|^2\ge a^2-2b^2$, we infer for any $\delta\in(0,1)$
	\begin{multline}
		|\nabla\vp_1-\nabla\vp_0|^2
		\ge (1-2\delta)|\nabla\vp_1-\nabla\vp_0|^2+\delta(1+\tfrac{1-u_0}{u_0})|\nabla\varphi_1|^2 
		\\- C_1(\iota)\delta\sum_{i=1}^n|\nabla\varphi_i-\nabla\varphi_0|^2
		-C_{T,\tilde u}|u-\tilde u|^2.
	\end{multline}
	For $\delta=\delta(\iota)>0$ small enough, we deduce
	\begin{align}
		G =	\sum_{i=1}^n|\nabla\vp_i-\nabla\vp_0|^2
		&\ge \epsilon_1(\iota) |\nabla\varphi_1|^2 -C_{T,\tilde u}|u-\tilde u|^2
		\\&\ge  \epsilon_2(\iota) |\nabla u_1-\nabla\tilde u_1|^2-C_{T,\tilde u}|u-\tilde u|^2,
	\end{align}
	where the last estimate follows from the definition of $\vp_1$ and elementary manipulations.
	To infer the full bound~\eqref{eq:108} one argues by symmetry. 
\end{remark}

\subsection{Exponential stability of steady states}\label{ssec:stab.ss}

\begin{proof}[Proof of Theorem~\ref{thm:stab.ss}]
	By Proposition~\ref{prop:edi}, the weak solution $u$ satisfies identity~\eqref{eq:edb}. Moreover, equation~\eqref{eq:weak} implies that $\int u_i(t)\,\dd x = \int \ol u_i\,\dd x $ for all $t\ge0$ and $i\in\{0,\dots,n\}$. Hence,  we have for all $0\le s<t<\infty$ and a positive constant $\epsilon=\epsilon(\Om)$
	\begin{align}
		\int_\Om h_\rel(u(t),\ol u)\,\dd x - \int_\Om h_\rel(u(s),\ol u)\,\dd x &= -\int_s^t\!\int_\Om \scrp(u)\,\dd x\dd\tau
		\\&\le -4\kappa\int_s^t\!\int_\Om \sum_{i=0}^n|\nabla\sqrt{u_i}|^2\,\dd x\dd\tau
		\\&\le -\epsilon\kappa\int_s^t\!\int_\Om h_\rel(u(\tau),\ol u)\,\dd x\dd\tau,
	\end{align}
	where the second step uses inequality~\eqref{eq:ffGrad} in Lemma~\ref{l:M.coerc} and 
	the last estimate follows from the logarithmic Sobolev inequality applied pointwise in time.
	A suitable version of Gronwall's inequality now implies that 
	\begin{align}
		\int_\Om h_\rel(u(t),\ol u)\,\dd x \le  \int_\Om h_\rel(u(0),\ol u)\,\dd x \,\exp(-\epsilon\kappa t)
	\end{align}
for all $t\ge0$.
	By a classical Csisz\'ar--Kullback type inequality (see e.g.~\cite{UAMT_2000}), the left-hand side is bounded below by $\frac{1}{C}\|u(t)-\ol u\|_{L^1(\Om)}^2$ for some $C=C(\Omega)\in(0,\infty)$. This proves the asserted inequality~\eqref{eq:stab.ss}.
\end{proof}

 \section{Existence and time asymptotics for partial interactions}\label{sec:ex}
 
 In this section, we perform a global weak existence analysis for systems where not necessarily all species are interacting.  
 Except for the case of full interactions, which is covered by~\cite[Theorem~2]{Juengel_2015} (cf.\ Theorem~\ref{thm:ex.full}), no systematic existence analysis is available in the literature for system~\eqref{eq:sys.all}. 
 Here, we take advantage of this result for full interactions, which easily gives rise to a family of approximate solutions to more general, partially interacting systems, and then perform a vanishing-interaction limit. The weak solutions obtained upon this construction further enjoy certain entropy dissipation inequalities that allow us to prove convergence to equilibrium as $t\to\infty$. 
Since the gradient bounds obtained from entropy dissipation may be quite poor, the asymptotic analysis involves some non-standard arguments when it comes to specifying a weak formulation of the equations and determining the long-time behaviour of solutions.

\subsection{Definitions and results}
In this section, we only assume in addition to~\ref{it:K.sp},~\ref{it:K.nn} the following hypothesis:
\begin{enumerate}[label=\textup{(H3)}]
	\item\label{it:ok} 	There exists $i_0\in\{0,\dots,n\}$ such that $K_{i_0j}>0$ for all $j\not=i_0$.
\end{enumerate}
Hypothesis~\ref{it:ok} is quite natural from a modelling point of view, and one may think of the species $X_{i_0}$ as representing the vacancies.

Depending on their interaction properties $\{K_{ij}\}$, the species $X_0,\dots, X_n$ differ in their (mathematical) character.  Our analysis suggests classifying them as follows.
\begin{definition}[Types \atype{}, \btype{} and \ctype{}]
	Let $\{K_{ij}\}_{i,j=0}^n$ satisfy~\ref{it:ok} and let $i\in\{0,\dots,n\}$.
	\begin{itemize}
		\item 
		We call the species $X_i$ of type~\atype{} if $K_{ij}>0$ for all $j\not=i$. 
		\item 	Let the species $X_i$ not be of type~\atype{}.	We say that $X_i$ is of type~\btype{} if 	
		$K_{i\ell}=0$ for all species $X_\ell$, $\ell\not=i$, that are not of type~\atype{}.
		\item A species that is neither of type \atype{} nor of type \btype{} will be referred to as type~\ctype{}.
	\end{itemize}
\end{definition}
We informally write $i\in\Atype$ to express the situation that $X_i$ is of type~\atype{} and use analogous notations for types \btype{} and \ctype.
Moreover, the symbol $\alpha$ will be reserved as an index for species of type~\atype{}, and
we frequently abbreviate
$$ \sum_{\alpha\in\Atype}u_\alpha:=\sum_{\{\alpha:\, X_\alpha\,\text{is of type }\Atype\}}u_\alpha.$$
With this convention we have $\{0,\dots,n\}=\Atype\,\dot\cup\,\Btype\,\dot\cup\,\Ctype$.

In terms of regularity, species of type \atype{} play a distinguished role. In view of Lemma~\ref{l:M.coerc} we expect their densities to have gradients in $L^2$.

We should further remark that hypothesis~\ref{it:ok} does not really make sense if one admits initial data $\uin$ satisfying $\uin_{i_0}\equiv0$. Indeed, in this case mass conservation and non-negativity enforce that $u_{i_0}\equiv0$ and problem~\eqref{eq:sys.all} effectively reduces to an $n$-species system that does not necessarily conform to~\ref{it:ok} any longer.
More generally, the present approach to cross-diffusion system \eqref{eq:sys.all} in the (more degenerate) partially interacting regime
primarily aims at understanding flows emanating from suitably regular initial data satisfying 
in particular $\inf_\Om\sum_{\alpha\in\Atype}\uin_\alpha>0$ or at least 
$\sum_{\alpha\in\Atype}\ol{\uin_\alpha}>0$; the latter condition turns out to be sufficient for the generalised solutions to be constructed below to relax to the steady state in the long-time limit (cf.\ Theorem~\ref{thm:conv.c}).
 
Let us now turn to the statement of our main result regarding existence for general interactions.
 Loosely speaking, we obtain global existence of weak solutions extending previous results in~\cite{BDiPS_2010,ZJ_2017}, where only species of type~\atype{} and~\btype{} were considered and where $\Atype=\{0\}$. (The method of proof in the present paper is, however, different.) 
 For the precise sense, in which gradients appearing in Theorem~\ref{thm:pex} below are generally to be understood, we refer to Remark~\ref{rem:defgrad} following the statement of this theorem.
 We further recall that $C([0,\infty);L^2_\mathrm{weak}(\Om))$ denotes the space of functions $v:[0,\infty)\to L^2(\Om)$ such that $t\mapsto v(t)$ is continuous with respect to the weak topology on $L^2(\Om)$.
 
 \begin{theorem}[Existence for general interactions]\label{thm:pex}
 	Let $\{K_{ij}\}_{i,j=0}^n$ satisfy~\ref{it:K.sp},~\ref{it:K.nn} and \ref{it:ok}.
 	Let further $\uin\in L^1(\Om)^\sn$ satisfy $\uin(x)\in\hsurfCl$ for a.e.\ $x\in\Om$.
 	There exists  $u=(u_i)_{i=0}^n\in L^\infty((0,\infty)\times\Om)^{1+n}\cap C([0,\infty);L^2_\mathrm{weak}(\Om))^{1+n}$
 	with $u(t,x)\in\hsurfCl$ for a.e.\ $(t,x)\in(0,\infty)\times\Om$
 	enjoying the regularity
 	\begin{itemize}
 		\item[] $\partial_tu_i\in L^2_\loc([0,\infty);H^1(\Om)^*)$ for all $i\in\{0,\dots,n\}$,
 		\item[] $\sqrt{u_\vd}\in L^2_\loc([0,\infty);H^1(\Om))$ for all $\vd\in\Atype$,
 		\item[]  $\sqrt{u_\vd u_i}\in L^2_\loc([0,\infty);H^1(\Om))$ for all $\vd\in\Atype$ and all $i\in\{0,\dots,n\}$,
 	\end{itemize}
 	taking the initial data $\uin$ and obeying for every $i\in\{0,\dots,n\}$ the conservation law
 	\begin{align}\label{eq:mcons}
 		\int_\Om u_i(t,x)\,\dd x=\int_\Om \uin_i(x)\,\dd x \text{ for all }t>0,
 	\end{align}
 	 and satisfying the cross-diffusion system~\eqref{eq:sys.all} in the following weak sense:
 	 
 	 for all $i,k\in\{0,\dots,n\}$ such that $K_{ik}>0$,
 	 there exists $\FF_{ik}\in L^2(0,\infty;L^2(\Om))^d$ with $\FF_{ik}=-\FF_{ki}$ and the properties that\footnote{See Remark~\ref{rem:defgrad} for the meaning of the terms $u_k\nabla u_i$.} 
 	 	\begin{subequations}\label{eq:Y}
 	 	\begin{align}\label{eq:Y.co}
 	 	\FF_{ik}&=u_k\nabla u_i - u_i\nabla u_k
 	 	\qquad\text{ a.e.\ in }\big\{\sum_{\vd\in\Atype} u_\vd>0\big\},
 	 \end{align}
 and if $\{i,k\}\cap\Atype\not=\emptyset$,
  \begin{align}\label{eq:Y.at}
  	\FF_{ik}&=u_k\nabla u_i - u_i\nabla u_k	\qquad\text{ a.e.\ in }(0,\infty)\times\Om,
  \end{align}
\end{subequations}
 	 such that for all $T<\infty$ and all $\psi\in C^\infty([0,T]\times\ol\Om)$
 	 \begin{align}\label{eq:ex.GEN}
 	 	\int_0^T\!\langle \partial_tu_i,\psi\rangle\,\dd t
 	 	=-\int_0^T\!\int_\Om \sum_{k}K_{ik}\mathbb{Y}_{ik}\cdot\nabla\psi	\,\dd x\dd t.
 	 \end{align}
  
  Furthermore, there exists a measurable set 
  $J\subset(0,\infty)$ with $\mathcal{L}^1(J)=0$ such that
   the following entropy dissipation inequalities hold for all $T\in(0,\infty):$
\begin{subequations}\label{eq:90.all}
  \begin{align}\label{eq:90}
  	&\mathcal{H}(u(T)) +  \int_0^T\!\!\int_\Om\!\scrp(u)\chi_{\{\sumA>0\}}\,\dd x\dd t
  	\le \mathcal{H}(\uin),
  	\\&
  	\begin{multlined}
  		\mathcal{H}(u(T)) +  \int_s^T\!\!\int_\Om\!\scrp(u)\chi_{\{\sumA>0\}}\,\dd x\dd t
  		\label{eq:90.s}\le \mathcal{H}(u(s))+C\mathcal{L}^d\big(\big\{\sumA(s)=0\big\}\big)
  		\\ \quad\text{ for all }s\in(0,T)\setminus J.
  	\end{multlined}
  \end{align}
\end{subequations}
Here, $\scrp(u)=\frac{1}{2}\sum_{i,j}K_{ij}u_iu_j|\frac{\nabla u_i}{u_i}-\frac{\nabla u_j}{u_j}|^2$ in ${\{\sumA>0\}}$
and $C<\infty$ is a fixed constant.
 \end{theorem}

\begin{remark}[Generalised Sobolev gradients]\label{rem:defgrad}
It is necessary to comment on the meaning of the gradients appearing in identity~\eqref{eq:Y.co} (resp.~\eqref{eq:90.all}), since for $i,k\not\in\Atype$ terms of the form $u_k\nabla u_i$ (resp.~$\sqrt{u_k}\nabla \sqrt{u_i}$) may not define a distribution in the classical sense. Let us first note that the regularity 
$\sqrt{u_\vd}\in L^2_\loc([0,\infty);H^1(\Om))$ and $\sqrt{u_\vd u_i}\in L^2_\loc([0,\infty);H^1(\Om))$ for $\vd\in\Atype$ ensures that  the products $\sqrt{u_\vd}\nabla \sqrt{u_i}$ and $u_\vd\nabla u_i$ are well-defined in the distributional sense and agree with the $L^2_{t,x}$ functions $\nabla(\sqrt{u_\vd u_i})-\sqrt{u_i}\nabla \sqrt{u_\vd}$ and $\nabla(u_\vd u_i) - u_i\nabla u_\vd$, respectively. This in turn determines measurable functions \enquote{ $\nabla \sqrt{u_i}$\,} and \enquote{ $\nabla u_i$\,} almost everywhere in the set $\{u_\vd>0\}$ and uniquely in the weighted Lebesgue space $L^2(u_\vd\dd x\dd t)$. The \enquote{maximal} set where for all $i\in\{0,\dots,n\}$ gradients $\nabla\sqrt{u_i},$ $ \nabla u_i$ can be defined in this sense is obtained by repeating the above reasoning for the sums $\sum_{\vd\in\Atype} \sqrt{u_\alpha}$ or $\sum_{\vd\in\Atype}u_\vd$, which allows us to define $\nabla \sqrt{u_i}$ and $\nabla u_i$ in a pointwise sense $\mathcal{L}^{1+d}$-a.e.\ in the set $\{\sum_{\vd\in\Atype}u_\vd>0\}$. 
Simple examples show that under the general hypotheses of Theorem~\ref{thm:pex}, we cannot expect gradients to exist on a set larger than this. In fact, letting $\Btype\not=\emptyset$, $\Ctype=\emptyset$ and $\uin_\vd\equiv0$ for all $\vd\in\Atype$, the function $u(t)\equiv \uin \in L^1(\Om)$ would be the natural candidate for a solution. Yet, under the hypotheses of Theorem~\ref{thm:pex}, the distributional gradient $\nabla \uin$ does not need to have a pointwise meaning.
If one focuses on the question of giving a meaning, on a set (possibly) larger than $\{\sumA>0\}$, to \emph{products} of the form $u_k\nabla u_i$ for $i,k\in\Ctype$, the problem becomes more delicate and is out of the scope of this manuscript. Here, we content ourselves with the fact that, as we will see in Theorem~\ref{thm:conv.c} below, the information contained in Theorem~\ref{thm:pex} suffices to uniquely determine the long-time asymptotic behaviour of such weak solutions, provided they emanate from data $\uin$ for which $\sum_{\alpha\in\Atype} \uin_\alpha$ is non-trivial.
\end{remark}
 
Let us note that the vector fields $\mathbb{Y}_{ik}$ in Theorem~\ref{thm:pex} will be obtained from a compactness argument (cf.~\eqref{eq:97.1}) and that, as discussed in Remark~\ref{rem:defgrad}, when $\{i,k\}\cap\Atype=\emptyset$, a direct description of $\mathbb{Y}_{ik}$ in terms of $u$ is only available in the set $\{\sum_{\vd\in\Atype}u_\vd>0\}$.
Observe, however, that for $i\in \Atype\cup\Btype$ identity~\eqref{eq:Y.at} implies that 
\begin{align}
	\sum_{k=0}^nK_{ik}\FF_{ik} = \sum_{j=0}^n A_{ij}(u)\nabla u_j \quad\text{ a.e.\ in }(0,\infty)\times\Om
\end{align}
with $A_{ij}(u)$ as in~\eqref{eq:defA},
so that in this case eq.~\eqref{eq:ex.GEN} agrees with the classical weak form of the $i$-th component of~\eqref{eq:sys.all}. 

 \begin{theorem}[Convergence to equilibrium]\label{thm:conv.c}
 	Assume the hypotheses of Theorem~\ref{thm:pex} and let $u\in L^\infty((0,\infty)\times\Om)^{1+n}\cap C([0,\infty);L^2_\mathrm{weak}(\Om))^{1+n}$  be an entropy-dissipating weak solution of~\eqref{eq:sys.all} in the sense that it enjoys
 	the properties listed in Theorem~\ref{thm:pex}.
 	Suppose further that $\sum_{\alpha\in \Atype}\ol{\uin_\alpha}>0$ with $\ol{\,\cdot\,}$ denoting the average $\frac{1}{|\Om|}\int_\Om\,\cdot\,$. 
 	Then 	
 		\begin{align}
 			\lim_{t\to\infty}\|u(t)-\ol{u}\|_{L^1(\Om)}=0,
 		\end{align}
where $\ol u\equiv\ol{\uin}$ is the steady state associated with the initial data $\uin$. 
\end{theorem}
 
 See Section~\ref{ssec:conv.c} for the proof of Theorem~\ref{thm:conv.c}.
Note that in models without \ctype{}-type species the vanishing of $\sum_{\alpha\in \Atype}\ol{\uin_\alpha}$ means that a trivial `solution' of~\eqref{eq:sys.all} can be obtained by letting $u\equiv\uin$.

 \subsection{Construction via vanishing-interaction limit}	
The purpose of this section is to prove Theorem~\ref{thm:pex}.
 
Observe that \atype{} is non-empty by hypothesis~\ref{it:ok}, and abbreviate
 \begin{align}
 K:=	\min_{\vd\in\Atype}\min_{i:i\not=\vd}K_{\vd i}>0.
 \end{align}  
 Our construction is based on taking a vanishing-interaction limit  in the family of models of full interaction type defined by $K_{ij}^{\ve}:=\max\{K_{ij},\ve\}$ for all $0\le i\not=j\le n$ (and as before $K_{ii}^{\ve}=0$ for all $i$), where $0<\ve\ll1$ is a small parameter which we assume to satisfy $\ve\le\min\{K_{ij}:K_{ij}>0\}$.
 Then, by Theorem~\ref{thm:ex.full}, there exists a global weak solution $u^\ve$ of system~\eqref{eq:sys.all} with parameters $\{K_{ij}\}:=\{K_{ij}^\ve\}$
 satisfying 
 $u^\ve(t,x)\in\hsurfCl$ for a.e.~$(t,x)$
 and enjoying the extra regularity (cf.~\eqref{eq:sqrtL2H1})
 \begin{align}
 	\sqrt{\uve_i}\in L^2_\loc([0,\infty);H^1(\Om))\quad\text{ for all }i\in\{0,\dots,n\}.
 \end{align}
  Thanks to Proposition~\ref{prop:edi} and Lemma~\ref{l:EPalg}, $\uve$
 obeys the entropy dissipation balance
\begin{align}\label{eq:edi.reg}
	\int_\Om h(\uve(T))\,\dd x + \int_s^T\!\int_\Om \scrp^\ve(\uve)\,\dd x\dd t
	= \int_\Om  h(\uve(s))\,\dd x
\end{align}
for all $0\le s<T<\infty$,
where 
\begin{align}
 	\scrp^\ve(\uve)= \frac{1}{2}\sum_{i,k}K_{ik}^\ve\uve_i\uve_k\big|\tfrac{\nabla\uve_i}{\uve_i} -\tfrac{\nabla\uve_k}{\uve_k}\big|^2
	=2\sum_{i,k}K_{ik}^\ve\big|\sqrt{\uve_k}\nabla\sqrt{\uve_i} -\sqrt{\uve_i}\nabla\sqrt{\uve_k}\big|^2.
\end{align}
Observe that, as a consequence of Lemma~\ref{l:M.coerc}, 
\begin{align}\label{eq:142}
		\scrp^\ve(\uve)&\gtrsim 
		\sum_{\vd\in\Atype}|\nabla\sqrt{\uve_\vd}|^2+\sum_{\vd\in\Atype}\sum_i\uve_\vd|\nabla\sqrt{\uve_i}|^2
		+\ve\sum_i|\nabla\sqrt{\uve_i}|^2
		\\&\qquad+\sum_{i,k}K_{ik}^\ve\big|\sqrt{\uve_k}\nabla\sqrt{\uve_i} -\sqrt{\uve_i}\nabla\sqrt{\uve_k}\big|^2+\sum_{\vd\in\Atype}\sum_i|\nabla\sqrt{\uve_\vd\uve_i}|^2,
\end{align}
where the bound for $\sum_{\vd\in\Atype}\sum_i|\nabla\sqrt{\uve_\vd\uve_i}|^2$ follows a posteriori by estimating $|\nabla\sqrt{\uve_\vd\uve_i}|^2\lesssim |\nabla\sqrt{\uve_\vd}|^2+\uve_\vd|\nabla\sqrt{\uve_i}|^2$.
Since $h\ge0$ and $\int_\Om  h(\uve(0))\,\dd x=\int_\Om  h(\uin)\,\dd x<\infty$, we infer the following $(\ve,T)$-uniform bounds:
 \begin{align}\label{eq:uni.bd}
 	\begin{aligned}
 	\int_0^T\!\int_\Om\bigg(\; &	\sum_{\vd\in\Atype}|\nabla\sqrt{\uve_\vd}|^2+\sum_{\vd\in\Atype}\sum_i\uve_\vd|\nabla\sqrt{\uve_i}|^2
 	+\sum_{\vd\in\Atype}\sum_i|\nabla\sqrt{\uve_\vd\uve_i}|^2
 	\\&+\sum_{i,k}K_{ik}^\ve\big|\sqrt{\uve_k}\nabla\sqrt{\uve_i} -\sqrt{\uve_i}\nabla\sqrt{\uve_k}\big|^2+\ve\sum_i|\nabla\sqrt{\uve_i}|^2\bigg)\,\dd x\dd t \le C<\infty.
 \end{aligned}
 \end{align}
Using the $\ve$-uniform control of the term $\sqrt{K_{ik}^\ve}(\sqrt{\uve_k}\nabla\sqrt{\uve_i} -\sqrt{\uve_i}\nabla\sqrt{\uve_k})$ in $L^2_{t,x}$ in~\eqref{eq:uni.bd} and the weak equation satisfied by $\uve_i$, we further deduce the $\ve$-uniform bound
\begin{align}\label{eq:144}
	\|\partial_t\uve_i\|_{L^2(0,T;H^1(\Om)^*)}\le C<\infty
\end{align}
for all $i\in\{0,\dots,n\}$.

Thanks to the above estimates and the control $u^\ve(t,x)\in\hsurfCl$, there exists a subsequence $\ve\downarrow0$  (not relabelled) such that 
for all $i\in\{0,\dots,n\}$ and suitable $u_i\in L^\infty((0,\infty)\times\Om)$, $0\le u_i\le 1$, 
satisfying $\sqrt{u_\vd}\in L^2_\loc([0,\infty);H^1(\Om))$ whenever $\vd\in\Atype$ the following convergence properties hold:
\begin{align}\label{eq:92}
	&\uve_i\wk u_i \quad\text { in } L^2_\loc([0,\infty);L^2(\Om)),
	\\&\partial_t\uve_i\weakstar\partial_tu_i\quad \text{ in }L^2_\loc([0,\infty);H^1(\Om)^*)\label{eq:92.2}
	\\&\uve_i(t)\wk u_i(t) \quad\text { in } L^2(\Om),\quad \text{ locally uniformly in }t\in[0,\infty),
	\label{eq:92.3}
\end{align}
where the last line exploits the fact that for every $T<\infty$, the family
$\{\uve_i\}$ is $\ve$-uniformly bounded in $L^\infty(0,T;L^2(\Om))\cap H^1(0,T;H^1(\Om)^*)$, giving compactness in $C([0,T];L^2_\mathrm{weak}(\Om))$ (cf.~\cite[Lemma~C.1]{Lions_1996}).

Additionally, the subsequence $\ve\downarrow 0$ will be chosen such that for all $\vd\in\Atype$ 
\begin{align}\label{eq:conv}
	\begin{aligned}
		&\sqrt{\uve_\vd}\wk \sqrt{u_\vd}\quad \text{ in }L^2_\loc([0,\infty);H^1(\Om)),
		\\&\uve_\vd\to u_\vd \quad \text{ a.e.\ in }(0,\infty)\times\Om.
	\end{aligned}
\end{align}

Since the constraint $\sum_{i=0}^n\uve_i = 1$ is linear, the weak convergence~\eqref{eq:92} implies that $\sum_{i=0}^nu_i(t,x)=1$ for a.e.\ $(t,x)\in(0,\infty)\times\Om$.
Line~\eqref{eq:92} further yields $\int_0^T\!\!\int_\Om u_i\vp\,\dd x\dd t\ge0$ for all
$\vp\in C_c((0,T)\times\Om)$ with $\vp\ge0$; hence $u_i\ge0$ a.e., allowing us to conclude
that  $u(t,x)\in \hsurfCl$ for a.e.\ $(t,x)\in(0,\infty)\times\Om$.

Next, we assert that for all $\alpha\in\Atype$ and $i\in\{0,\dots,n\}$
\begin{align}\label{eq:101}
	{\uve_\vd}\uve_i\to u_\vd u_i\quad  \text{ in }L^2_\loc([0,\infty);L^2(\Om)).
\end{align}
The convergence~\eqref{eq:101} follows from the Aubin--Lions lemma applied to ${\uve_\vd}\uve_i$. To verify the required control on the time derivative, we compute in the sense of distributions
\begin{align}
	\partial_t\big({\uve_\vd}\uve_i\big) 
	= \partial_t\uve_\vd\uve_i + \partial_t\uve_i\uve_\vd.
\end{align}
Let now $T<\infty$. For the first term on the right-hand side we compute for $\psi\in C^\infty([0,T]\times\ol\Om)$ using the equation satisfied by $\uve_\vd$
\begin{align}
	\int_0^T\!\langle\partial_t\uve_\vd\uve_i,\psi\rangle\,\dd t= \int_0^T\!\langle\partial_t\uve_\vd,\uve_i\psi\rangle\,\dd t
	&=-\int_0^T\!\int_\Om\sum_jA_{\vd j}(\uve)\nabla \uve_j\cdot\nabla (\uve_i\psi)\,\dd x\dd t
	\\&=- \int_0^T\!\langle F^\ve,\sqrt{\uve_\vd}\nabla (\uve_i\psi)\rangle\,\dd t,
\end{align}
where $F^\ve:=\frac{1}{\sqrt{\uve_\vd}}\sum_jA_{\vd j}(\uve)\nabla \uve_j\in L^2_{t,x}$.
By~\eqref{eq:uni.bd} the term $F^\ve$ is $\ve$-uniformly bounded in $L^2(0,T;L^2(\Om))$. 
Hence, recalling also the uniform bounds
\begin{align}
	\|\uve_\vd\|_{L^2([0,T];H^1(\Om))}+
	\|\sqrt{\uve_\vd}\nabla\uve_i\|_{L^2(0,T;L^2(\Om))}+\|\partial_t\uve_i\|_{L^2(0,T;H^1(\Om)^*)}\le C_T<\infty,
\end{align}
we find that 
$\partial_t\big({\uve_\vd}\uve_i\big)$ is $\ve$-uniformly bounded in 
$$L^1(0,T;(H^1\cap L^\infty)(\Om)^*).$$ 
Observing further the $\ve$-uniform bound of 
$\nabla\big({\uve_\vd}\uve_i\big)$ in $L^2_\loc([0,\infty);L^2(\Om))$ (cf.~\eqref{eq:uni.bd}), we infer~\eqref{eq:101} by appealing to the Aubin--Lions lemma (see e.g.~\cite[Lemma 7.7]{Roubicek_2013_npde}).

Thus, upon extraction of a subsequence, we may henceforth assume that
\begin{align}\label{eq:99}
	{\uve_\vd}\uve_i\to u_\vd u_i\quad  \text{a.e.\ in }(0,\infty)\times\Om.
\end{align}
We next assert that this convergence implies pointwise convergence a.e.\ of $\uve_i$ to $u_i$ in the set $V_\vd:=\{u_\vd>0\}$, where for definiteness we let $u_\vd(t),u_i(t)$ denote the precise representatives of the corresponding Lebesgue classes. To show this assertion, let us fix $N\subset(0,\infty)\times\Om$ measurable with $\mathcal{L}^{1+d}(N)=0$ in such a way that 
$\uve_\vd\to u_\vd$ and	${\uve_\vd}\uve_i\to u_\vd u_i$ pointwise in $((0,\infty)\times\Om)\setminus N$. In particular, for every $(t,x)\in V_\vd$, 
we have $\uve_\vd(t,x)>0$ for $\ve>0$ small enough, allowing us to infer that 
$u_\vd\uve_i=\frac{u_\vd}{\uve_\vd}\cdot{\uve_\vd}\uve_i\to u_\vd u_i$  pointwise in $V_\vd\setminus N$, since for each $(t,x)\in V_\vd\setminus N$ the expression $\uve_\vd\uve_i\cdot\frac{u_\vd}{\uve_\vd}$ is well-defined for small enough $\ve$.
We have thus shown that for all $i\in\{0,\dots,n\}$
\begin{align}\label{eq:98}
	\uve_i\to u_i \text{ pointwise a.e.\ in }\bigcup_{\vd\in\Atype}\{u_\vd>0\}
	=\Big\{\sum_{\vd\in\Atype}u_\vd>0\Big\}.
\end{align}

\begin{lemma}\label{l:mix.lim}For all $i,k\in\{0,\dots,n\}$ with $K_{ik}>0$
	there exist vector fields $\WW_{ik},\FF_{ik}\in L^2(0,\infty;L^2(\Om))^d$, $\WW_{ik}=-\WW_{ki}$, $\FF_{ik}=-\FF_{ki}$
	 that satisfy
		\begin{align}\label{eq:97.2}
		 \WW_{ik}&=\sqrt{u_k}\nabla\sqrt{u_i}-\sqrt{u_i}\nabla \sqrt{u_k}
\hspace{-2cm}
		&&\text{ a.e.\ in }\big\{\sum_{\vd\in\Atype} u_\vd>0\big\},
		\\\FF_{ik}&=u_k\nabla u_i - u_i\nabla u_k
	&&\text{ a.e.\ in }\big\{\sum_{\vd\in\Atype} u_\vd>0\big\},
	\label{eq:97.3}
	\end{align}
and if $\{i,k\}\cap\Atype\not=\emptyset$,
	\begin{align}\label{eq:97.2'}
	\WW_{ik}&=\sqrt{u_k}\nabla\sqrt{u_i}-\sqrt{u_i}\nabla \sqrt{u_k}
	\hspace{-2cm}
	&&\text{ a.e.\ in }(0,\infty)\times\Om,
	\\\FF_{ik}&=u_k\nabla u_i - u_i\nabla u_k
	&&\text{ a.e.\ in }(0,\infty)\times\Om,
	\label{eq:97.3'}
\end{align}
and are such that 
		\begin{align}\label{eq:97}
		\sqrt{\uve_k}\nabla\sqrt{\uve_i}-\sqrt{\uve_i}\nabla\sqrt{\uve_k}
		\;\rightharpoonup\; &\WW_{ik} \hspace{-2cm}&&\text{ in } L^2((0,\infty)\times\Om),\qquad
		\\\label{eq:97.1}
		\uve_k\nabla\uve_i-\uve_i\nabla\uve_k
		\;\rightharpoonup\; &\FF_{ik} &&\text{ in } L^2((0,\infty)\times\Om).\qquad
	\end{align}
Moreover, for all $i\in\{0,\dots,n\}$ one has 
	\begin{align}\label{eq:95}
	\ve\nabla\sqrt{\uve_i}\;\to\;0\quad \text{ in } L^2((0,\infty)\times\Om).
\end{align}
\end{lemma}	

\begin{proof}
Properties~\eqref{eq:98},~\eqref{eq:conv} and~\eqref{eq:uni.bd} imply that for all $\vd\in\Atype$
\begin{align}\label{eq:121x}
	\sqrt{\uve_i}\nabla \sqrt{\uve_\vd}\;\wk\; \sqrt{u_i}\nabla \sqrt{u_\vd}\;\;\;\text{ in }
	L^2((0,\infty)\times\Om)\text{ for all }i\in\{0,\dots,n\}.
\end{align}
Furthermore, using~\eqref{eq:98} and~\eqref{eq:conv}, it is easy to see for all $\vd\in\Atype$ that
$\sqrt{\uve_\vd}\nabla \sqrt{\uve_i}$ converges to $\sqrt{u_\vd}\nabla \sqrt{u_i}$ in the sense of distributions and hence, by~\eqref{eq:uni.bd}, 
\begin{align}\label{eq:121}
	\sqrt{\uve_\vd}\nabla \sqrt{\uve_i}\;\wk\; \sqrt{u_\vd}\nabla \sqrt{u_i}\;\;\;\text{ in }
	L^2((0,\infty)\times\Om)\text{ for all }i\in\{0,\dots,n\}.
\end{align}

Let us now prove~\eqref{eq:97},~\eqref{eq:97.2}.	Since $K_{ik}>0$,  the term 
	$\WW^\ve_{ik}:=(\sqrt{\uve_k}\nabla\sqrt{\uve_i}-\sqrt{\uve_i}\nabla\sqrt{\uve_k})$ is $\ve$-uniformly bounded in $L^2((0,\infty)\times\Om)$. Hence, by weak compactness in $L^2((0,\infty)\times\Om)$, there exists $\WW_{ik}\in L^2((0,\infty)\times\Om)$ such that 
	along a subsequence $\ve\downarrow0$, $\WW^\ve_{ik}\rightharpoonup \WW_{ik}$ in $L^2((0,\infty)\times\Om)$. Thus,
	\begin{align}
			\eta\big(\sum_{\vd\in\Atype}\uve_\vd\big)\WW^\ve_{ik}\wk 	\eta\big(\sum_{\vd\in\Atype}u_\vd\big) \WW_{ik}
	\end{align}
for all $\eta\in C^\infty([0,1])$ with $\supp\eta\subset(0,1]$.
	At the same time, for any such $\eta$ we have weakly in $L^2_{t,x}$ the convergence
	\begin{multline}
			\eta\big(\sum_{\vd\in\Atype}\uve_\vd\big)\WW^\ve_{ik}=	\tfrac{\eta\big(\sum_{\vd\in\Atype}\uve_\vd\big)}{\sum_{\vd\in\Atype}\uve_\vd}
			\sum_{\vd\in\Atype}(\sqrt{\uve_k}\uve_\vd\nabla\sqrt{\uve_i}
			-\sqrt{\uve_i}\uve_\vd\nabla\sqrt{\uve_k})
		\\\wk
		\tfrac{\eta\big(\sum_{\vd\in\Atype}u_\vd\big)}{\sum_{\vd\in\Atype}u_\vd}
		\sum_{\vd\in\Atype}(\sqrt{u_k}u_\vd\nabla\sqrt{u_i}
		-\sqrt{u_i}u_\vd\nabla\sqrt{u_k})
		\\=\eta\big(\sum_{\vd\in\Atype}u_\vd\big)( \sqrt{u_k}\nabla \sqrt{u_i}- \sqrt{u_i} \nabla \sqrt{u_k}),
	\end{multline}
where we also used~\eqref{eq:121}.
	Uniqueness of the weak limit in $L^2$ implies that 
	\begin{align}
		\eta\big(\sum_{\vd\in\Atype}u_\vd\big)  \WW_{ik}
		=\eta\big(\sum_{\vd\in\Atype}u_\vd\big)( \sqrt{u_k}\nabla\sqrt{u_i}- \sqrt{u_i}\nabla \sqrt{u_k}).
	\end{align}
Since $\eta\in C^\infty_c((0,1])$ was arbitrary, we infer that
 $\WW_{ik}=\sqrt{u_k}\nabla \sqrt{u_i}- \sqrt{u_i} \nabla \sqrt{u_k}$ a.e.\ in $\{\sum_{\vd\in\Atype}u_\vd>0\}$. Note that the antisymmetry of $\WW$ follows from that of $\WW^\ve$.
 If $\{i,k\}\cap\Atype\not=\emptyset$, then~\eqref{eq:121x},~\eqref{eq:121} and uniqueness of the weak limit in $L^2$ immediately yield 
 $\WW_{ik}=\sqrt{u_k}\nabla\sqrt{u_i}-\sqrt{u_i}\nabla \sqrt{u_k}$ a.e.\ in $(0,\infty)\times\Om.$
 This proves the first assertions~\eqref{eq:97},~\eqref{eq:97.2},~\eqref{eq:97.2'}. 
 The assertions~\eqref{eq:97.1},~\eqref{eq:97.3},~\eqref{eq:97.3'}  concerning $\FF_{ik}$ can be shown along similar lines.

The convergence~\eqref{eq:95} follows from the fact that $\sqrt{\ve}\nabla\sqrt{\uve_i}$ is $\ve$-uniformly bounded in $L^2_{t,x}$ (cf.~\eqref{eq:uni.bd}).
\end{proof}

We are now in a position to complete the proof of Theorem~\ref{thm:pex}.

\begin{proof}[Proof of Theorem~\ref{thm:pex}]
	Let $u$ denote the limiting candidate constructed on page~\pageref{eq:92} (see \eqref{eq:92}--\eqref{eq:conv}). We have already seen that it enjoys the asserted regularity properties and satisfies $u(t,x)\in\hsurfCl$ for a.e.\ $(t,x)$.
	
	The convergence properties~\eqref{eq:92.2},~\eqref{eq:97.1} and~\eqref{eq:95} allow to pass to the limit $\ve\downarrow0$ in the weak form satisfied by $\uve_i$, $i\in\{0,\dots,n\}$, for any test function $\psi$, giving the asserted equation~\eqref{eq:ex.GEN} with~$\FF_{ik}=-\FF_{ki}$ satisfying~\eqref{eq:Y.co} and, if $\{i,k\}\cap\Atype\not=\emptyset$,~\eqref{eq:Y.at} (cf.~\eqref{eq:97.3},~\eqref{eq:97.3'}).
	
	We next show the conservation law~\eqref{eq:mcons}. It is a consequence of the convergence
	\begin{align}
		\partial_t\uve_i\weakstar \partial_tu_i\quad\text{ in }L^2_\loc([0,\infty);H^1(\Om)^*),
	\end{align}
the fact that $\int_0^t\langle \partial_t\uve_i,1\rangle\dd \tau\equiv0$, which follows from the equation for $\uve_i$, 	and the identity\footnote{Identity~\eqref{eq:122} can be shown by approximation, using the density of the set $C^1([0,T];L^2(\Om))$ in the Sobolev-Bochner space $\{v\in L^2(0,T;L^2(\Om)):\frac{\dd v}{\dd t}\in L^2(0,T;H^1(\Om)^*)\}$ equipped with the standard norm.}
	\begin{align}\label{eq:122}
		\int_0^t\langle \partial_tu_i,1\rangle\dd \tau = \int_\Om u_i(t,x)\,\dd x-\int_\Om u_i(0,x)\,\dd x.
	\end{align}
	
	It remains to prove the entropy dissipation inequalities~\eqref{eq:90.all}.
	The proof of inequality \eqref{eq:90} being somewhat simpler, we only provide details on how to derive~\eqref{eq:90.s}, where we rely on the entropy balance law~\eqref{eq:edi.reg} for $\uve$.
		For passing to the limit in the entropy at the final time $T$ it is convenient to use~\eqref{eq:92.3}, which gives, thanks to the weak lower semicontinuity in $L^2(\Om)$ of the entropy functional, $\mathcal{H}(u(T))\le \liminf_{\ve\to0}\mathcal{H}(\uve(T))$. Note that here we have specified $u=u(T)$ on a set of times of measure zero in such a way that~\eqref{eq:92.3} holds for all $T\in[0,\infty)$.
		To deal with the entropy at times $s\in(0,T)$, we observe that thanks to~\eqref{eq:98} there exists a measurable set $J\subset[0,\infty)$ with $\mathcal{L}^1(J)=0$ such that for all $t\in[0,\infty)\setminus J$
		\begin{align}\label{eq:70}
			\uve(t)\to u(t)\qquad\text{ a.e.\ in }\big\{\sumA(t) >0\big\}.
		\end{align} 
		Since $|\uve|\lesssim 1$, Lebesgue's dominated convergence theorem thus gives for $s\not\in J$
		\begin{align}
			\limsup_{\ve\to0}\int_\Om  h(\uve(s))\,\dd x
			\le \int_\Om h(u(s))\,\dd x + C\mathcal{L}^d\big(\big\{\sumA(s)=0\big\}\big),
		\end{align}
		where $C<\infty$ is a universal constant.
		
		We are left with estimating below the liminf of the dissipation term. First recall that
		\begin{align}
			\scrp^\ve(\uve)\ge \scrp(\uve) = 2\sum_{i,k}K_{ik} |\sqrt{\uve_k}\nabla\sqrt{\uve_i}-\sqrt{\uve_i}\nabla\sqrt{\uve_k}|^2
			= 2\sum_{i,k}K_{ik} |\WW_{ik}^\ve|^2,
		\end{align}
	where for $K_{ik}>0$ we have $\WW_{ik}^\ve:= \sqrt{\uve_k}\nabla\sqrt{\uve_i}-\sqrt{\uve_i}\nabla\sqrt{\uve_k}\wk \WW_{ik}$ in $L^2_{t,x}$ with $\WW_{ik}$ as in Lemma~\ref{l:mix.lim}. Thus, weak lower semicontinuity of the norm in $L^2_{t,x}$ yields
		\begin{align}
			\int_0^T\!\!\int_\Om\!\scrp(u)\chi_{\{\sumA>0\}}\,\dd x\dd t
				\le \int_0^T\!\!\int_\Om\!2\sum_{i,k}K_{ik} |\WW_{ik}|^2\,\dd x\dd t
			\le \liminf_{\ve\to0}	\int_0^T\!\!\int_\Om\!\scrp^\ve(\uve)\,\dd x\dd t.
		\end{align}	
	The proof of Theorem~\ref{thm:pex} is now complete. 
\end{proof}

\subsection{Convergence to equilibrium}\label{ssec:conv.c}

\begin{proof}[Proof of Theorem~\ref{thm:conv.c}]Below, the mass conservation property~\eqref{eq:mcons} will, in places, be used without explicit reference.
Let $J$ be such that~\eqref{eq:90.all} holds. Thanks to~\eqref{eq:90} there exists a sequence $\{t_k\}\subset [0,\infty)\setminus J$ with $t_k\to\infty$ such that
\begin{align}\label{eq:72}
	\lim_{k\to\infty}\int_\Om\scrp(u(t_k))\chi_{\{\sum_{\alpha\in\Atype}u_\alpha(t_k)>0\}}\,\dd x=0.
\end{align}
Similarly to~\eqref{eq:142}, we infer from Lemma~\ref{l:M.coerc} the uniform bound
\begin{align}\label{eq:142.ve=0}
	\scrp(u)\chi_{\{\sum_{\alpha\in\Atype}u_\alpha>0\}}&\gtrsim 
	\sum_{\vd\in\Atype}|\nabla\sqrt{u_\vd}|^2+\sum_{\vd\in\Atype}\sum_i|\nabla\sqrt{u_\vd u_i}|^2
	\\&\qquad+\sum_{i,l}K_{il}\big|\sqrt{u_l}\nabla\sqrt{u_i} -\sqrt{u_i}\nabla\sqrt{u_l}\big|^2\chi_{\{\sum_{\alpha\in\Atype}u_\alpha>0\}},
\end{align}
where we recall that the gradients in the last line are to be understood as explained in Remark~\ref{rem:defgrad}.

Hence, \eqref{eq:72} implies that $\lim_{k\to\infty}\|\nabla\sqrt{u_\vd}(t_k)\|_{L^2(\Om)}=0$ for all $\vd\in\Atype$,
which further yields, by classical properties for Sobolev weak derivatives,
\begin{align}
	\sqrt{u_\alpha}(t_k) \to \sqrt{\ol{u}_\alpha}\qquad\text{ in }H^1(\Om).
\end{align}
We assert that this further gives
\begin{align}\label{eq:71}
	u_i(t_k) \to \ol{u_i}\qquad \text{ in }L^2(\Om)\qquad \text{for all }i\in\{0,\dots,n\}.
\end{align}
To show~\eqref{eq:71}, pick $\hat\vd\in\Atype$ such that $\ol{u}_{\hat\vd}>0$ (such $\hat\vd$ exists by hypothesis) and observe that~\eqref{eq:72} combined with~\eqref{eq:142.ve=0} implies the convergence
\begin{align}
	\chi_{\{u_{\hat\vd}(t_k)>0\}}\big(\sqrt{u_{\hat\vd}}\nabla \sqrt{u_i}- \sqrt{u_i} \nabla \sqrt{u_{\hat\vd}}\big)|_{t_k} \;\overset{k\to\infty}{\to} \;0\qquad\text{in }L^2(\Om),
\end{align}
and hence $	\chi_{\{u_{\hat\vd}(t_k)>0\}}\sqrt{u_{\hat\vd}(t_k)}\nabla \sqrt{u_i(t_k)}=\sqrt{u_{\hat\vd}(t_k)}\nabla \sqrt{u_i(t_k)}\to0$ in $L^2(\Om)$.
Recalling that
$\sqrt{u_{\hat\vd}(t_k)}\nabla \sqrt{u_i(t_k)}:=\nabla(\sqrt{u_{\hat\vd}(t_k)}\sqrt{u_i(t_k)})
-(\nabla\sqrt{u_{\hat\vd}(t_k)}) \sqrt{u_i(t_k)}$ (cf.\ Remark~\ref{rem:defgrad}) and using again classical arguments for Sobolev weak derivatives, we find that 
\begin{align}
	\sqrt{u_{\hat\vd}(t_k)u_i(t_k)}\to c\equiv\text{const}\qquad\text{ in }H^1(\Om).
\end{align}
An elementary argument similar to that leading to~\eqref{eq:98} finally gives, upon possibly passing to a subsequence, $u_i(t_k)\to\ol{u}_i$ a.e.\ in $\Om$, which implies~\eqref{eq:71}.

Let now $\{\tau_j\}\subset(0,\infty)$ with $\tau_j\to\infty$ be arbitrary.
Then there exists a subsequence $\{t_{k_j}\}$ of $\{t_k\}$  (with $k_j$ not necessarily strictly increasing, but $\lim_{j\to\infty}k_j=\infty$) and $j_0\in\mathbb{N}$ such that $t_{k_j}<\tau_j$ for all $j\ge j_0$.
From the strong entropy dissipation inequality~\eqref{eq:90.s} (and mass conservation) we infer for all $\tau>0$ and all $s\in(0,\tau)\setminus J$
\begin{align}\label{eq:91}
	\mathcal{H}_\rel(u(\tau),\ol u)\le \mathcal{H}_\rel(u(s),\ol u) + C\mathcal{L}^d\big(\big\{\sumA(s)=0\big\}\big).
\end{align}
Inserting $\tau:=\tau_j$ and $s:=t_{k_j}$ in inequality \eqref{eq:91} and sending $j\to\infty$ yields
\begin{align}
	\lim_{j\to\infty}\mathcal{H}_\rel(u(\tau_j),\ol u)=0.
\end{align}
Here we used the strong convergence~\eqref{eq:71} and the hypothesis $\sum_{\alpha\in\mathfrak A} \ol u_\alpha>0$ to ensure that the right-hand side of~\eqref{eq:91} vanishes as $s=t_{k_j}\to\infty$.
\end{proof}

\appendix

\section{An auxiliary estimate}\label{app:aux.est}

The estimate asserted in the lemma below has previously been employed in~\cite[Theorem 2.8, Case $\mathcal{A}_+$]{Hopf_prepr2021} in a similar context. For the purpose of being self-contained, we here also recall its elementary proof.
\begin{lemma}[Cf.\ \cite{Hopf_prepr2021}]
	Let $G\subset\mathbb{R}^N$ be convex, let $f\in C^3_b(G)$ and let $v\in H^1(\Om)^N, \tilde v\in C^{0,1}(\ol\Om,G)$. Then, for a.e.\ $x\in \Om$ such that $v(x)\in G$
	\begin{multline}
			\big| \nabla\big(f(v)-f(\tilde v)-\sum_{j=1}^N D_{j}f(\tilde v)(v_j-\tilde v_j)\big)\big|
			\\\lesssim_{\|f\|_{C^3(G)}}|\nabla v-\nabla \tilde v||v-\tilde v| +\|\nabla\tilde v\|_{L^\infty} |v-\tilde v|^2\quad\text{at the point }x,
	\end{multline}
where we abbreviated $\nabla f(v):=\sum_{k=1}^ND_kf(v)\nabla v_k$.
\end{lemma}
\begin{proof}
	Using the definition of $\nabla f(v)$ and rearranging terms, we may rewrite
	{\small
	\begin{align}
		&\nabla\big(f(v)-f(\tilde v)-\sum_{j=1}^N D_{j}f(\tilde v)(v_j-\tilde v_j)\big) 
		\\
		&=	\sum_{k=1}^N \big(D_kf(v)\nabla v_k-D_kf(\tilde v)\nabla\tilde v_k\big)
		-\sum_{j=1}^ND_jf(\tilde v)(\nabla v_j-\nabla\tilde v_j)
		-\sum_{j,l=1}^ND_{lj}f(\tilde v)(v_j-\tilde v_j)\nabla \tilde v_l
		\\&=	\sum_{k=1}^N (D_kf(v)-D_kf(\tilde v))(\nabla v_k-\nabla\tilde v_k)
		\\&\hspace{.25\linewidth} +\sum_{l=1}^N\big(D_lf(v)-D_lf(\tilde v)-\sum_{j=1}^ND_jD_{l}f(\tilde v)(v_j-\tilde v_j)\big)\nabla \tilde v_l.
	\end{align}
\normalsize} Notice that, in the last step, some indices have been relabelled. 
The assertion now follows by means of the Cauchy--Schwarz inequality and an application of Taylor's theorem to $D_lf\in C^2_b(G)$, $l\in\{1,\dots, N\}$.
\end{proof}

\section{Notations}\label{app:not}

For definiteness, we here provide some details concerning our notations, most of which are quite classical.
\begin{itemize}
	\item The convention $K_{ii}:=0$ for all $i\in\{0,\dots,n\}$ is adopted throughout the paper.
	\item We will further adopt the convention that $K_{ij}\cdot Q\equiv0$ whenever $K_{ij}=0$, even if the quantity `$Q$' is not well-defined mathematically. This will simplify notations as it avoids having to restrict sums involving $K_{ij}$ to pairs $\{i,j\}$ satisfying $K_{ij}>0$.
	\item Summation: since we are dealing not only with $(n+1)$ but also with $n$ dimensional systems,  we avoid the usual summation convention. Unspecified sums involving a Latin summation index (such as $i,j,k$) always indicate summing over all indices, that is, we let $ \sum_i:=\sum_{i=0}^n$ etc. Likewise 
	$\sum_{i\not=j}:=\sum_{i\in\{0,\dots,n\}}\sum_{j\in\{0,\dots,n\}\setminus\{i\}}$ and 
	$\sum_{j:j\not=i}:=\sum_{j\in\{0,\dots,n\}\setminus\{i\}}$.
	\item We further let 
	$\min_{i\not=j}=\min_{i\in\{0,\dots,n\}}\min_{j\in\{0,\dots,n\}\setminus\{i\}}$ and $\min_{i:i\not=j}=\min_{i\in\{0,\dots,n\}\setminus\{j\}}$.
	\item The notation $A(u)\nabla u$ for $u=u(x)\in \mathbb{R}^{1+n}$ and  $A(u)\in \mathbb{R}^{(1+n)\times (1+n)}$ will be used as a short hand for the $d\times (1+n)$ matrix $(A(u)\nabla u)_i=\sum_{j}A_{ij}(u)\nabla u_j$. Analogous short-hand notations will be adopted for  the matrices $\MM(\uvec),\hA(U)\in \mathbb{R}^{n\times n}$.
	\item For an interval $I\subseteq[0,\infty)$, $C(I;L^2_\mathrm{weak}(\Om))$ denotes the space of functions $v:I\to L^2(\Om)$ such that $t\mapsto v(t)$ is weakly continuous in $L^2(\Om)$. 
	\item  For non-negative quantities $A,B$, we write $A\lesssim B$
	to indicate that there exists a constant $C\in(0,\infty)$ that only depends on fixed parameters
	such that $A\le CB$.
	We define $A\gtrsim B$ by $B\lesssim A$, and
	$A\sim B$ by [$A\lesssim B$ and $A\gtrsim B$]. 
	\item For a smooth manifold $M$ and $p\in M$, we denote by $T_pM$ the tangent space of $M$ at $p$.
	\item For a sufficiently smooth function $f=f(u), u\in \mathbb{R}^N,$ and $i\in\{1,\dots,n\}$, we write $D_if:=\frac{\partial}{\partial u_i}f$. Analogous notations are used if $f=f(u_0,\dots,u_N)$ etc. 
	We further abbreviate $D_{ik}f:=D_kD_if.$
	\item Unless specified otherwise, we abbreviate $\ol u_i:=\frac{1}{|\Om|}\int_\Om u_i\,\dd x$ for all $i$.
\end{itemize}	
	
%\bibliographystyle{abbrv}
%\bibliography{sizeex}

\end{document}